\documentclass[a4paper,british,a4wide]{scrartcl}
\usepackage[T1]{fontenc}
\usepackage[latin9]{inputenc}
\usepackage{babel}
\usepackage{verbatim}
\usepackage{prettyref}
\usepackage{mathtools}
\usepackage{amsmath}
\usepackage{amsthm}
\usepackage{amssymb}
\usepackage[unicode=true,pdfusetitle,
 bookmarks=true,bookmarksnumbered=false,bookmarksopen=false,
 breaklinks=false,pdfborder={0 0 1},backref=false,colorlinks=false]
 {hyperref}

\makeatletter

\pdfpageheight\paperheight
\pdfpagewidth\paperwidth

\theoremstyle{definition}
\newtheorem*{defn*}{\protect\definitionname}
\theoremstyle{plain}
\newtheorem{thm}{\protect\theoremname}[section]
\theoremstyle{plain}
\newtheorem{prop}[thm]{\protect\propositionname}
\theoremstyle{remark}
\newtheorem{rem}[thm]{\protect\remarkname}
\theoremstyle{definition}
\newtheorem{example}[thm]{\protect\examplename}
\theoremstyle{plain}
\newtheorem{lem}[thm]{\protect\lemmaname}

\@ifundefined{date}{}{\date{}}
\usepackage{prettyref}

\usepackage{enumerate}

\newcommand*{\e}{\mathrm{e}}
\renewcommand*{\i}{\mathrm{i}}
\newcommand{\m}{\operatorname{m}}

\newcommand{\R}{\mathbb{R}}
\newcommand{\N}{\mathbb{N}}
\renewcommand{\C}{\mathbb{C}}
\newcommand{\dom}{\operatorname{dom}}
\newcommand{\ran}{\operatorname{ran}}
\newcommand{\spt}{\operatorname{spt}}
\renewcommand{\d}{\,\mathrm{d}}
\renewcommand{\Re}{\operatorname{Re}}

\newcommand{\dive}{\operatorname{div}}
\newcommand{\grad}{\operatorname{grad}}
\renewcommand{\tilde}{\widetilde}
\newcommand{\IV}{\operatorname{IV}}

\newcommand{\His}{\operatorname{His}}

\newrefformat{prop}{Proposition \ref{#1}}
\newrefformat{lem}{Lemma \ref{#1}}
\newrefformat{thm}{Theorem \ref{#1}}
\newrefformat{cor}{Corollary \ref{#1}}
\newrefformat{rem}{Remark \ref{#1}}
\newrefformat{exa}{Example \ref{#1}}
\newrefformat{sec}{Section \ref{#1}}
\newrefformat{sub}{Subsection \ref{#1}}
\newrefformat{eq}{(\ref{#1})}

\theoremstyle{definition}

\newrefformat{hyp}{Hypotheses \ref{#1}}

\makeatother

\providecommand{\definitionname}{Definition}
\providecommand{\examplename}{Example}
\providecommand{\lemmaname}{Lemma}
\providecommand{\propositionname}{Proposition}
\providecommand{\remarkname}{Remark}
\providecommand{\theoremname}{Theorem}

\begin{document}
\title{Semigroups and Evolutionary Equations\thanks{Based on parts of the authors Habilitation thesis \cite{Trostorff_habil}.}}
\author{Sascha Trostorff\thanks{Mathematisches Seminar, CAU Kiel, Germany, email: trostorff@math.uni-kiel.de}}

\maketitle
\textbf{Abstract. }We show how strongly continuous semigroups can
be associated with evolutionary equations. For doing so, we need to
define the space of admissible history functions and initial states.
Moreover, the initial value problem has to be formulated within the
framework of evolutionary equations, which is done by using the theory
of extrapolation spaces. The results are applied to two examples.
First, differential-algebraic equations in infinite dimensions are
treated and it is shown, how a $C_{0}$-semigroup can be associated
with such problems. In the second example we treat a concrete hyperbolic
delay equation.\medskip{}

\textbf{Keywords: }Evolutionary equations, $C_{0}$-semigroups, admissible
history and initial value\textbf{ \medskip{}
}

\textbf{2010 MSC: }46N20, 47D06, 35F16 \textbf{\medskip{}
}

\section{Introduction}

In this article we bring together two theories for dealing with partial
differential equations: the theory of $C_{0}$-semigroups on the one
hand and the theory of evolutionary equations on the other hand. In
particular, we show how $C_{0}$-semigroups can be associated with
a given evolutionary equation. \\
The framework of evolutionary equations was introduced in the seminal
paper \cite{Picard}. Evolutionary equations are equations of the
form 
\begin{equation}
\left(\partial_{t}M(\partial_{t})+A\right)U=F,\label{eq:evo}
\end{equation}
where $\partial_{t}$ denotes the temporal derivative, $M(\partial_{t})$
is a bounded operator in space-time defined via a functional calculus
for $\partial_{t}$ and $A$ is an, in general, unbounded spatial
operator. The function $F$ defined on $\R$ and taking values in
some Hilbert space is a given source term and one seeks for a solution
$U$ of the above equation. Here, the notion of solution is quite
weak, since one just requires that the solution should belong to some
exponentially weighted $L_{2}$-space. Thus, all operators have to
be introduced in these spaces. Especially, the time derivative is
introduced as an unbounded normal operator on such a space and so,
in order to solve \prettyref{eq:evo}, one has to deal with the sum
of two unbounded operators ($\partial_{t}$ and $A$). Problems of
the form \prettyref{eq:evo} cover a broad spectrum of different types
of differential equations, such as hyperbolic, parabolic, elliptic
and mixed-type problems, integro-differential equations \cite{Trostorff2012_integro},
delay equations \cite{Kalauch2011} and fractional differential equations
\cite{Picard2013_fractional}. Also, generalisations to some nonlinear
\cite{Trostorff2012_NA,Trostorff2012_nonlin_bd} and non-autonomous
problems \cite{Picard2013_nonauto,Trostorff2013_nonautoincl,Waurick2015_nonauto,Trostorff2018}
are possible. The solution theory is quite easy and just relies on
pure Hilbert space theory.\\
On the other hand, there is the well-established theory of $C_{0}$-semigroups
dealing with so-called Cauchy problems (see e.g. \cite{hille1957functional,Pazy1983,engel2000one}).
These are abstract equations of the form 
\begin{align}
(\partial_{t}+A)U & =F,\nonumber \\
U(0) & =U_{0},\label{eq:Cauchy}
\end{align}
where $A$ is a suitable operator acting on some Banach space. Although,
\prettyref{eq:Cauchy} just seems to be a special case of \prettyref{eq:evo}
for $M(\partial_{t})=1$, the theories are quite different. While
we focus on solutions lying in $L_{2}$ in the theory of evolutionary
equations, one seeks for continuous solutions in the framework of
$C_{0}$-semigroups. Moreover, while \prettyref{eq:evo} holds on
$\R$ as time horizon, \prettyref{eq:Cauchy} just holds on $\R_{\geq0}$
and is completed by an initial condition. Thus, in order to associate
a $C_{0}$-semigroup with equations of the form \prettyref{eq:evo}
one has to find a way to formulate initial value problems and then
derive assumptions, which would yield the additional regularity for
the solutions (namely continuity with respect to time). This is the
purpose of this work.\\
As we have indicated above, equations of the form \prettyref{eq:evo}
also cover delay equations, where it is more natural to prescribe
histories instead of an initial state at time $0$. Moreover, \prettyref{eq:evo}
also covers so-called differential algebraic equations (see \cite{Mehrmann2006}
for the finite-dimensional case and \cite{Trostorff_DAE2019,Trostorff_DAE_higherindex_2018,Trostorff2019_DAE_semigroup}
for infinite dimensions), where not every element of the underlying
state space can be used as an initial state. Thus, one is confronted
with the problem of defining the `right' initial values and histories
for \prettyref{eq:evo} depending on the operators involved. Moreover,
one has to incorporate these initial conditions within the framework
of evolutionary equations, that is, initial conditions should enter
the equation as a suitable source term on the right-hand side. This
can be done by using extrapolation spaces and by extending the solution
theory to those. Then it will turn out that initial conditions can
be formulated by distributional right hand sides, which belong to
a suitable extrapolation space associated with the time derivative
operator $\partial_{t}$. Having the right formulation of initial
value problems at hand, one can associate a $C_{0}$-semigroup on
a product space consisting of the current state in the first and the
past of the unknown in the second component. This idea was already
used to deal with delay equations within the theory of $C_{0}$-semigroups,
see \cite{Batkai_2005}. As it turns out, this product space is not
closed (as a subspace of a suitable Hilbert space) and in order to
extend the associated $C_{0}$-semigroup to its closure one needs
to impose similar conditions as in the Hille-Yosida Theorem. The key
result, which will be used to extend the semigroup is the theorem
of Widder-Arendt (see \cite{Arendt1987} or \prettyref{thm:WA} below).\\
The paper is structured as follows: We begin by recalling the basic
notions and well-posedness results for evolutionary problems (\prettyref{sec:Evolutionary-Problems})
and for extrapolation spaces (\prettyref{sec:Extrapolation-spaces}).
Then, in order to formulate initial value problems within the framework
of evolutionary equations, we introduce a cut-off operator as an unbounded
operator on the extrapolation space associated with the time derivative
and discus some of its properties (\prettyref{sec:Cut-off-operators}).
\prettyref{sec:Admissible-histories-for} is then devoted to determine
the `right' space of admissible histories and initial values for
a given evolutionary problem. We note here that we restrict ourselves
to homogeneous problems in the sense that we do not involve an additional
source term besides the given history. The main reason for that is
that such source terms would restrict and change the set of admissible
histories, a fact which is well-known in the theory of differential-algebraic
equations. In \prettyref{sec:-semigroups-associated-with} we associate
a $C_{0}$-semigroup on the before introduced product space of admissible
initial values and histories and prove the main result of this article
(\prettyref{thm:HY}). In the last section we discuss two examples.
First, we apply the results to abstract differential algebraic equations
and thereby re-prove the Theorem of Hille-Yosida as a special case.
In the second example, we discuss a concrete hyperbolic delay equation
and prove that we can associate a $C_{0}$-semigroup with this problem.
\\
Throughout, every Hilbert space is assumed to be complex and the inner
product $\langle\cdot,\cdot\rangle$ is conjugate-linear in the first
and linear in the second argument.

\section{Evolutionary Problems\label{sec:Evolutionary-Problems}}

We recall the basic notions and results for evolutionary problems,
as they were introduced in \cite{Picard} (see also \cite[Chapter 6]{Picard_McGhee}).
We begin by the definition of the time derivative operator on an exponentially
weighted $L_{2}$-space (see also \cite{picard1989hilbert}).
\begin{defn*}
Let $\rho\in\R$ and $H$ a Hilbert space. We set 
\[
L_{2,\rho}(\R;H)\coloneqq\{f:\R\to H\,;\,f\text{ measurable},\,\int_{\R}\|f(t)\|^{2}\e^{-2\rho t}\d t\}
\]
with the common identification of functions coinciding almost everywhere.
Then $L_{2,\rho}(\R;H)$ is a Hilbert space with respect to the inner
product
\[
\langle f,g\rangle_{\rho}\coloneqq\int_{\R}\langle f(t),g(t)\rangle\e^{-2\rho t}\d t\quad(f,g\in L_{2,\rho}(\R;H)).
\]
Moreover, we define the operator 
\[
\partial_{t,\rho}:H_{\rho}^{1}(\R;H)\subseteq L_{2,\rho}(\R;H)\to L_{2,\rho}(\R;H),\;f\mapsto f'
\]
where 
\[
H_{\rho}^{1}(\R;H)\coloneqq\{f\in L_{2,\rho}(\R;H)\,;\,f'\in L_{2,\rho}(\R;H)\}
\]
with $f'$ denoting the usual distributional derivative.
\end{defn*}
We recall some facts on the operator $\partial_{t,\rho}$ and refer
to \cite{Kalauch2011} for the respective proofs.
\begin{prop}
\label{prop:td}Let $\rho\in\R$ and $H$ a Hilbert space.

\begin{enumerate}[(a)]

\item The operator $\partial_{t,\rho}$ is densely defined, closed
and linear and $C_{c}^{\infty}(\R;H)$ is a core for $\partial_{t,\rho}$.

\item The spectrum of $\partial_{t,\rho}$ is given by 
\[
\sigma(\partial_{t,\rho})=\{\i t+\rho\,;\,t\in\R\}.
\]

\item For $\rho\ne0$ the operator $\partial_{t,\rho}$ is boundedly
invertible with $\|\partial_{t,\rho}^{-1}\|=\frac{1}{|\rho|}$ and
the inverse is given by 
\[
\left(\partial_{t,\rho}^{-1}f\right)(t)=\begin{cases}
\int_{-\infty}^{t}f(s)\d s & \text{ if }\rho>0,\\
-\int_{t}^{\infty}f(s)\d s & \text{ if }\rho<0
\end{cases}
\]
for $f\in L_{2,\rho}(\R;H)$ and $t\in\R$.

\item The operator $\partial_{t,\rho}$ is normal with $\partial_{t,\rho}^{\ast}=-\partial_{t,\rho}+2\rho.$

\item The following variant of Sobolev's embedding theorem holds:
\[
H_{\rho}^{1}(\R;H)\hookrightarrow C_{\rho}(\R;H)
\]
continuously, where 
\[
C_{\rho}(\R;H)\coloneqq\{f:\R\to H\,;\,f\text{ continuous, }\sup_{t\in\R}\|f(t)\|\e^{-\rho t}<\infty\}.
\]

\end{enumerate}
\end{prop}

As a normal operator, $\partial_{t,\rho}$ possesses a natural functional
calculus, which can be described via the so-called Fourier-Laplace
transform.
\begin{defn*}
Let $\rho\in\R$ and $H$ a Hilbert space. We denote by $\mathcal{L}_{\rho}$
the unitary extension of the mapping 
\[
C_{c}(\R;H)\subseteq L_{2,\rho}(\R;H)\to L_{2}(\R;H),\;f\mapsto\left(t\mapsto\frac{1}{\sqrt{2\pi}}\int_{\R}\e^{-(\i t+\rho)s}f(s)\d s\right).
\]
\end{defn*}
\begin{rem}
Note that for $\rho=0,$ the operator $\mathcal{L}_{0}$ is nothing
but the classical Fourier transform, which is unitary due to Plancharel's
Theorem (see e.g. \cite[Theorem 9.13]{rudin1987real}). Since $\mathcal{L}_{\rho}=\mathcal{L}_{0}\exp(-\rho\cdot)$
with
\[
\exp(-\rho\cdot):L_{2,\rho}(\R;H)\to L_{2}(\R;H),\;f\mapsto\left(t\mapsto f(t)\e^{-\rho t}\right)
\]
for $t\in\R$, it follows that $\mathcal{L}_{\rho}$ is unitary as
a composition of unitary operators.
\end{rem}

\begin{prop}[{\cite[Corollary 2.5]{Kalauch2011}}]
 Let $\rho\in\R$ and $H$ a Hilbert space. We define the operator
$\m$ by 
\begin{align*}
\m & :\dom(\m)\subseteq L_{2}(\R;H)\to L_{2}(\R;H),\;f\mapsto\left(t\mapsto tf(t)\right)
\end{align*}
with maximal domain 
\[
\dom(\m)\coloneqq\{f\in L_{2}(\R;H)\,;\,(t\mapsto tf(t))\in L_{2}(\R;H)\}.
\]
Then 
\[
\partial_{t,\rho}=\mathcal{L}_{\rho}^{\ast}(\i\m+\rho)\mathcal{L}_{\rho}.
\]
\end{prop}

Using the latter proposition, we can define an operator-valued functional
calculus for $\partial_{t,\rho}$ as follows.
\begin{defn*}
Let $\rho\in\R$ and $H$ a Hilbert space. Let $F:\{\i t+\rho\,;\,t\in\R\}\to L(H)$
be strongly measurable and bounded. Then we define
\[
F(\partial_{t,\rho})\coloneqq\mathcal{L}_{\rho}^{\ast}F(\i\m+\rho)\mathcal{L}_{\rho}\in L(L_{2,\rho}(\R;H)),
\]
where 
\[
F(\i\m+\rho)f\coloneqq\left(t\mapsto F(\i t+\rho)f(t)\right)\quad(f\in L_{2}(\R;H)).
\]
\end{defn*}
An important class of operator-valued function of $\partial_{t,\rho}$
are those functions yielding causal operators. 
\begin{prop}[{\cite[Theorem 6.1.1, Theorem 6.1.4]{Picard_McGhee}}]
\label{prop:material law} Let $\rho_{0}\in\R$ and $H$ a Hilbert
space. If $M:\C_{\Re>\rho_{0}}\to L(H)$ is analytic and bounded,
then $M(\partial_{t,\rho})$ is causal for each $\rho>\rho_{0},$
i.e., for $f\in L_{2,\rho}(\R;H)$ with $\spt f\subseteq\R_{\geq a}$
for some $a\in\R$ it follows that 
\[
\spt M(\partial_{t,\rho})f\subseteq\R_{\geq a}.
\]
Moreover, $M(\partial_{t,\rho})$ is independent of the choice of
$\rho>\rho_{0}$ in the sense that 
\[
M(\partial_{t,\rho})f=M(\partial_{t,\mu})f\quad(f\in L_{2,\rho}(\R;H)\cap L_{2,\mu}(\R;H))
\]
for each $\rho,\mu>\rho_{0}$.
\end{prop}

\begin{rem}
\begin{enumerate}[(a)]

\item The proof of causality is based on a theorem by Paley and Wiener,
which charcterises the functions in $L_{2}(\R_{\geq0};H)$ in terms
of their Laplace transform (see \cite{Paley_Wiener} or \cite[19.2 Theorem ]{rudin1987real}).
The independence of $\rho$ is a simple application of Cauchy's Theorem
for analytic functions.

\item It is noteworthy that causal, translation-invariant and bounded
operators are always of the form $M(\partial_{t,\rho})$ for some
analytic and bounded mapping defined on a right half plane (see \cite{Foures1955,Weiss1991}). 

\end{enumerate}
\end{rem}

Finally, we are in the position to define well-posed evolutionary
problems.
\begin{defn*}
Let $\rho_{0}\in\R$ and $H$ a Hilbert space. Moreover, let $M:\C_{\Re>\rho_{0}}\to L(H)$
be analytic and bounded and $A:\dom(A)\subseteq H\to H$ densely defined,
closed and linear. Then we call an equation of the form 
\[
(\partial_{t,\rho}M(\partial_{t,\rho})+A)u=f
\]
the \emph{evolutionary equation associated with $(M,A).$ }The problem
is called \emph{well-posed} if there is $\rho_{1}>\rho_{0}$ such
that $zM(z)+A$ is boundedly invertible for each $z\in\C_{\Re\geq\rho_{1}}$
and 
\[
\C_{\Re\geq\rho_{1}}\ni z\mapsto(zM(z)+A)^{-1}
\]
is bounded. Moreover we set $s_{0}(M,A)$ as the infimum over all
such $\rho_{1}>\rho_{0}$. 
\end{defn*}
\begin{thm}
\label{thm:well_posed}Let $\rho_{0}\in\R$ and $H$ a Hilbert space.
Moreover, let $M:\C_{\Re>\rho_{0}}\to L(H)$ be analytic and bounded
and $A:\dom(A)\subseteq H\to H$ densely defined closed and linear.
We assume that the evolutionary equation associated with $(M,A)$
is well-posed. Then $\overline{\partial_{t,\rho}M(\partial_{t,\rho})+A}$
is boundedly invertible as an operator on $L_{2,\rho}(\R;H)$ for
each $\rho>s_{0}(M,A)$. Moreover, the inverse 
\[
S_{\rho}\coloneqq\left(\overline{\partial_{t,\rho}M(\partial_{t,\rho})+A}\right)^{-1}
\]
is causal and independent of the choice of $\rho>s_{0}(M,A)$ in the
sense of \prettyref{prop:material law}. %
\end{thm}

\begin{proof}
We note that the operator $\overline{\partial_{t,\rho}M(\partial_{t,\rho})+A}$
for $\rho>s_{0}(M,A)$ is unitarily equivalent to the multiplication
operator on $L_{2}(\R;H)$ associated with the operator-valued function
\[
F(t)\coloneqq\left(\i t+\rho\right)M(\i t+\rho)+A,
\]
see \cite[Lemma 2.2]{Trostorff2015_secondorder}, which is boundedly
invertible by assumption. The causality and independence of $\rho$
are an immediate consequence of \prettyref{prop:material law}, since
$S_{\rho}=N(\partial_{t,\rho})$ for the analytic and bounded function
$N(z)\coloneqq(zM(z)+A)^{-1}$ for $z\in\C_{\Re>s_{0}(M,A)}$. 
\end{proof}

\section{Extrapolation spaces\label{sec:Extrapolation-spaces}}

In this section we recall the notion of extrapolation spaces associated
with a boundedly invertible operator on some Hilbert space $H$. We
refer to \cite[Section 2.1]{Picard_McGhee} for the proof of the results
presented here. 
\begin{defn*}
Let $C:\dom(C)\subseteq H\to H$ be a densely defined, closed, linear
and boundedly invertible operator on some Hilbert space $H$. We define
the Hilbert space 
\[
H^{1}(C)\coloneqq\dom(C)
\]
equipped with the inner product 
\[
\langle x,y\rangle_{H^{1}(C)}\coloneqq\langle Cx,Cy\rangle\quad(x,y\in\dom(C)).
\]
Moreover, we set 
\[
H^{-1}(C)\coloneqq H^{1}(C^{\ast})',
\]
the dual space of $H^{1}(C^{\ast})$. 
\end{defn*}
\begin{rem}
Another way to introduce the space $H^{-1}(C)$ is taking the completion
of $H$ with respect to the norm 
\[
x\mapsto\|C^{-1}x\|.
\]
\end{rem}

\begin{prop}[{\cite[Theorem 2.1.6]{Picard_McGhee}}]
 Let $C:\dom(C)\subseteq H\to H$ be a densely defined, closed, linear
and boundedly invertible operator on some Hilbert space $H$. Then
$H^{1}(C)\hookrightarrow H\hookrightarrow H^{-1}(C)$ with dense and
continuous embeddings. Here, the second embedding is given by 
\[
H\to H^{-1}(C),\;x\mapsto\left(\dom(C^{\ast})\ni y\mapsto\langle x,y\rangle\right).
\]
Moreover, the operator 
\[
C:H^{1}(C)\to H
\]
is unitary and 
\[
C:\dom(C)\subseteq H\to H^{-1}(C)
\]
possesses a unitary extension, which will again be denoted by $C$.
\end{prop}

\begin{example}
Let $\rho\ne0$ and $H$ a Hilbert space. Then we set 
\begin{align*}
H_{\rho}^{1}(\R;H) & \coloneqq H^{1}(\partial_{t,\rho}),\\
H_{\rho}^{-1}(\R;H) & \coloneqq H^{-1}(\partial_{t,\rho}).
\end{align*}
Moreover, the Dirac distribution $\delta_{t}$ at a point $t\in\R$
belongs to $H_{\rho}^{-1}(\R;\C)$ and 
\[
\partial_{t,\rho}^{-1}\delta_{t}=\begin{cases}
\e^{2\rho t}\chi_{\R_{\geq t}} & \text{ if }\rho>0,\\
-\e^{2\rho t}\chi_{\R_{\leq t}} & \text{ if }\rho<0.
\end{cases}
\]
Indeed, for $\rho>0$ we have that 
\begin{align*}
\langle\partial_{t,\rho}\chi_{\R_{\geq t}},\varphi\rangle_{H_{\rho}^{-1}(\R;\C)\times H_{\rho}^{1}(\R;\C)} & =\int_{t}^{\infty}\left(\partial_{t,\rho}^{\ast}\varphi\right)(s)\e^{-2\rho s}\d s\\
 & =-\int_{t}^{\infty}\left(\varphi\e^{-2\rho\cdot}\right)'(s)\d s\\
 & =\varphi(t)\e^{-2\rho t}
\end{align*}
for each $\varphi\in C_{c}^{\infty}(\R;\C)$, which shows the asserted
formula. The statement for $\rho<0$ follows by the same rationale. 
\end{example}

\begin{prop}
\label{prop:extension_sol_op} Let $\rho_{0}\geq0$ and $H$ a Hilbert
space. Moreover, let $M:\C_{\Re>\rho_{0}}\to H$ be analytic and bounded
and $A:\dom(A)\subseteq H\to H$ densely defined, linear and closed
such that the evolutionary problem associated with $(M,A)$ is well-posed.
Then for each $\rho>s_{0}(M,A)$ we obtain 
\[
S_{\rho}[H_{\rho}^{1}(\R;H)]\subseteq H_{\rho}^{1}(\R;H)
\]
and 
\[
S_{\rho}:L_{2,\rho}(\R;H)\subseteq H_{\rho}^{-1}(\R;H)\to H_{\rho}^{-1}(\R;H)
\]
is bounded and thus has a unique bounded extension to the whole $H_{\rho}^{-1}(\R;H).$ 
\end{prop}

\begin{proof}
The assertion follows immediately by realising that 
\[
\left(\partial_{t,\rho}M(\partial_{t,\rho})+A\right)\partial_{t,\rho}\subseteq\partial_{t,\rho}\left(\partial_{t,\rho}M(\partial_{t,\rho})+A\right).\tag*{\qedhere}
\]
\end{proof}
We recall that for a densely defined, closed, linear operator $A:\dom(A)\subseteq H_{0}\to H_{1}$
between two Hilbert spaces $H_{0}$ and $H_{1}$, the operators $A^{\ast}A$
and $AA^{\ast}$ are selfadjoint and positive. Then the moduli of
$A$ and $A^{\ast}$ are defined by 
\[
|A|\coloneqq\sqrt{A^{\ast}A},\quad|A^{\ast}|\coloneqq\sqrt{AA^{\ast}}
\]
and are selfadjoint positive operators, too (see e.g. \cite[Theorem 7.20]{Weidmann}). 
\begin{prop}[{\cite[Lemma 2.1.16]{Picard_McGhee}}]
Let $H_{0},H_{1}$ be Hilbert spaces and $A:\dom(A)\subseteq H_{0}\to H_{1}$
densely defined, closed and linear. Then 
\[
A:\dom(A)\subseteq H_{0}\to H^{-1}(|A^{\ast}|+1)
\]
is bounded and hence, possesses a bounded extension to $H_{0}$. 
\end{prop}

\section{Cut-off operators\label{sec:Cut-off-operators}}

The main goal of the present section is to extend the cut-off operators
$\chi_{\R_{\geq t}}$ and $\chi_{\R_{\leq t}}$ for some $t\in\R$
defined on $L_{2,\rho}(\R;H)$ to the extrapolation space $H_{\rho}^{-1}(\R;H).$
For doing so, we start with the following observation.
\begin{lem}
Let $\rho>0,t\in\R$ and $H$ be a Hilbert space. We define the operators
\begin{align*}
\chi_{\R_{\geq t}}(\m) & :L_{2,\rho}(\R;H)\to L_{2,\rho}(\R;H),\quad f\mapsto\left(s\mapsto\chi_{\R_{\geq t}}(s)f(s)\right),\\
\chi_{\R_{\leq t}}(\m) & :L_{2,\rho}(\R;H)\to L_{2,\rho}(\R;H),\quad f\mapsto\left(s\mapsto\chi_{\R_{\leq t}}(s)f(s)\right).
\end{align*}
Then for $f\in L_{2,\rho}(\R;H)$ we have\footnote{Note that $\partial_{t,\rho}^{-1}f$ has a continuous representer
by \prettyref{prop:td} (e).} 
\begin{align*}
\chi_{\R_{\geq t}}(\m)f & =\partial_{t,\rho}\chi_{\R_{\geq t}}(\m)\partial_{t,\rho}^{-1}f-\e^{-2\rho t}\left(\partial_{t,\rho}^{-1}f\right)(t+)\delta_{t},\\
\chi_{\R_{\leq t}}(\m)f & =\partial_{t,\rho}\chi_{\R_{\leq t}}(\m)\partial_{t,\rho}^{-1}f+\e^{-2\rho t}\left(\partial_{t,\rho}^{-1}f\right)(t-)\delta_{t}.
\end{align*}
\end{lem}

\begin{proof}
We just prove the formula for $\chi_{\R_{\geq t}}(\m)$. So, let $f\in L_{2,\rho}(\R;H)$
and set $F\coloneqq\partial_{t,\rho}^{-1}f.$ We recall from \prettyref{prop:td}
(c) that 
\[
F(t)=\int_{-\infty}^{t}f(s)\d s\quad(t\in\R).
\]
For $g\in C_{c}^{\infty}(\R;H)$ we compute 
\begin{align*}
 & \langle\partial_{t,\rho}\chi_{\R_{\geq t}}(\m)\partial_{t,\rho}^{-1}f,g\rangle_{H^{-1}(\partial_{t,\rho})\times H^{1}(\partial_{t,\rho}^{\ast})}\\
 & =\langle\chi_{\R_{\geq t}}(\m)\partial_{t,\rho}^{-1}f,\partial_{t,\rho}^{\ast}g\rangle_{L_{2,\rho}(\R;H)}\\
 & =\int_{t}^{\infty}\langle F(s),-g'(s)+2\rho g(s)\rangle\e^{-2\rho s}\d s\\
 & =\int_{t}^{\infty}\langle f(s),g(s)\rangle\e^{-2\rho s}\d s+F(t+)g(t)\e^{-2\rho t}\\
 & =\langle\chi_{\R_{\geq t}}(\m)f,g\rangle_{L_{2,\rho}(\R;H)}+\langle\e^{-2\rho t}F(t+)\delta_{t},g\rangle_{H^{-1}(\partial_{t,\rho})\times H^{1}(\partial_{t,\rho}^{\ast}).}
\end{align*}
Since $C_{c}^{\infty}(\R;H)$ is dense in $H^{1}(\partial_{t,\rho}^{\ast})$
by \prettyref{prop:td} (a), we derive the asserted formula.
\end{proof}
The latter representation of the cut-off operators on $L_{2,\rho}(\R;H)$
leads to the following definition on $H_{\rho}^{-1}(\R;H).$ 
\begin{defn*}
Let $\rho>0$ and $H$ a Hilbert space. For $t\in\R$ we define the
operators 
\begin{align*}
P_{t} & :\dom(P_{t})\subseteq H_{\rho}^{-1}(\R;H)\to H_{\rho}^{-1}(\R;H),\\
Q_{t} & :\dom(Q_{t})\subseteq H_{\rho}^{-1}(\R;H)\to H_{\rho}^{-1}(\R;H),
\end{align*}
with the domains 
\begin{align*}
\dom(P_{t}) & \coloneqq\{f\in H_{\rho}^{-1}(\R;H)\,;\,(\partial_{t,\rho}^{-1}f)(t+)\text{ exists}\},\\
\dom(Q_{t}) & \coloneqq\{f\in H_{\rho}^{-1}(\R;H)\,;\,(\partial_{t,\rho}^{-1}f)(t-)\text{ exists}\}
\end{align*}
by 
\[
P_{t}f\coloneqq\partial_{t,\rho}\chi_{\R_{\geq t}}(\m)\partial_{t,\rho}^{-1}f-\e^{-2\rho t}\left(\partial_{t,\rho}^{-1}f\right)(t+)\delta_{t}\quad(f\in\dom(P_{t}))
\]
and 
\[
Q_{t}f\coloneqq\partial_{t,\rho}\chi_{\R_{\leq t}}(\m)\partial_{t,\rho}^{-1}f+\e^{-2\rho t}\left(\partial_{t,\rho}^{-1}f\right)(t-)\delta_{t}\quad(f\in\dom(Q_{t})).
\]
\end{defn*}
\begin{rem}
For a function $f\in L_{1,\mathrm{loc}}(\R;H)$ we say that $a\coloneqq f(t+)$
for some $t\in\R$ if 
\[
\forall\varepsilon>0\:\exists\delta>0:\;\lambda\left(\{s\in[t,t+\delta[\,;\,|f(s)-a|>\varepsilon\right)=0,
\]
where $\lambda$ denotes the Lebesgue measure on $\R.$ The expression
$f(t-)$ is defined analogously. 
\end{rem}

We conclude this section by some properties of the so introduced cut-off
operators. 
\begin{prop}
\label{prop:properties_cut_off}Let $H$ be a Hilbert space, $\rho>0$,
$y\in H$ and $s,t\in\R.$ Then the following statements hold.

\begin{enumerate}[(a)]

\item $\delta_{s}y\in\dom(P_{t})$ and 
\[
P_{t}\delta_{s}y=\begin{cases}
\delta_{s}y & \text{ if }s>t,\\
0 & \text{ if }s\leq t.
\end{cases}
\]

\item For $f\in\dom(P_{t})\cap\dom(Q_{t})$ we obtain 
\[
f=P_{t}f+Q_{t}f+\e^{-2\rho t}\left(\left(\partial_{t,\rho}^{-1}f\right)(t+)-\left(\partial_{t,\rho}^{-1}f\right)(t-)\right)\delta_{t}.
\]

\item For $f\in H_{\rho}^{-1}(\R;H)$ we have $\spt f\subseteq\R_{\leq t}$
if and only if $f\in\ker(P_{t}).$ Here, the support $\spt f$ is
meant in the sense of distributions.

\end{enumerate}
\end{prop}

\begin{proof}
\begin{enumerate}[(a)]

\item We note that $\partial_{t,\rho}^{-1}\delta_{s}y=\e^{2\rho s}\chi_{\R_{\geq s}}y$
and hence, $\delta_{s}\in\dom(P_{t}).$ Moreover, 
\[
P_{t}\delta_{s}y=\partial_{t,\rho}\chi_{\R_{\geq t}}(\m)\chi_{\R_{\geq s}}y\e^{2\rho s}-\e^{-2\rho t}\left(\e^{2\rho s}\chi_{\R_{\geq s}}y\right)(t+)\delta_{t}=\begin{cases}
\delta_{s}y & \text{ if }s>t,\\
0 & \text{ if }s\leq t.
\end{cases}
\]

\item If $f\in\dom(P_{t})\cap\dom(Q_{t})$ we compute 
\begin{align*}
P_{t}f+Q_{t}f & =\partial_{t,\rho}\chi_{\R_{\geq t}}(\m)\partial_{t,\rho}^{-1}f-\e^{-2\rho t}\left(\partial_{t,\rho}^{-1}f\right)(t+)\delta_{t}+\partial_{t,\rho}\chi_{\R_{\leq t}}(\m)\partial_{t,\rho}^{-1}f+\e^{-2\rho t}\left(\partial_{t,\rho}^{-1}f\right)(t-)\delta_{t}\\
 & =\partial_{t,\rho}\partial_{t,\rho}^{-1}f-\e^{-2\rho t}\left(\left(\partial_{t,\rho}^{-1}f\right)(t+)-\left(\partial_{t,\rho}^{-1}f\right)(t-)\right)\delta_{t}\\
 & =f-\e^{-2\rho t}\left(\left(\partial_{t,\rho}^{-1}f\right)(t+)-\left(\partial_{t,\rho}^{-1}f\right)(t-)\right)\delta_{t}.
\end{align*}

\item Let $f\in H_{\rho}^{-1}(\R;H)$ and assume first that $\spt f\subseteq\R_{\leq t}$.
We first prove that $\partial_{t,\rho}^{-1}f$ is constant on $\R_{\geq t}$.
For doing so, we define 
\[
V\coloneqq\{\chi_{\R_{\geq t}}x\,;\,x\in H\}\subseteq L_{2,\rho}(\R;H).
\]
Then $V$ is a closed subspace and for $g\in L_{2,\rho}(\R;H)$ we
have that 
\[
g\in V^{\bot}\quad\Leftrightarrow\quad\int_{t}^{\infty}g(s)\e^{-2\rho s}\d s=0.
\]
 For $g\in L_{2,\rho}(\R;H)$ we obtain 
\[
\langle\chi_{\R_{\geq t}}(\m)\partial_{t,\rho}^{-1}f,g\rangle_{L_{2,\rho}(\R;H)}=\langle f,\left(\partial_{t,\rho}^{\ast}\right)^{-1}\chi_{\R_{\geq t}}(\m)g\rangle_{H_{\rho}^{-1}(\R;H)\times H_{\rho}^{1}(\R;H)}
\]
 and an elementary computation shows 
\[
\left(\left(\partial_{t,\rho}^{\ast}\right)^{-1}\chi_{\R_{\geq t}}(\m)g\right)(s)=\int_{s}^{\infty}\chi_{\R_{\geq t}}(r)g(r)\e^{2\rho(s-r)}\d r\quad(s\in\R).
\]
Consequently, for $g\in V^{\bot}$ we infer that $\left(\partial_{t,\rho}^{\ast}\right)^{-1}\chi_{\R_{\geq t}}(\m)g=0$
on $\R_{\leq t}.$ Hence, $\langle\chi_{\R_{\geq t}}(\m)\partial_{t,\rho}^{-1}f,g\rangle_{L_{2,\rho}(\R;H)}=0$
for each $g\in V^{\bot}$ and thus, $\chi_{\R_{\geq t}}(\m)\partial_{t,\rho}^{-1}f\in V$,
which proves that $\partial_{t,\rho}^{-1}f$ is constant on $\R_{\geq t}.$
In particular, this shows $f\in\dom(P_{t})$ and 
\begin{align*}
P_{t}f & =\partial_{t,\rho}\chi_{\R_{\geq t}}(\m)\partial_{t,\rho}^{-1}f-\e^{-2\rho t}\left(\partial_{t,\rho}^{-1}f\right)(t+)\delta_{t}\\
 & =\partial_{t,\rho}\chi_{\R_{\geq t}}\left(\partial_{t,\rho}^{-1}f\right)(t+)-\e^{-2\rho t}\left(\partial_{t,\rho}^{-1}f\right)(t+)\delta_{t}\\
 & =0.
\end{align*}
Assume on the other hand that $f\in\ker(P_{t})$ and let $\varphi\in C_{c}^{\infty}(\R_{>t};H).$
We then compute, using that $\spt\partial_{t,\rho}^{\ast}\varphi\subseteq\R_{>t}$
\begin{align*}
\langle f,\varphi\rangle_{H_{\rho}^{-1}(\R;H)\times H_{\rho}^{1}(\R;H)} & =\langle\partial_{t,\rho}^{-1}f,\partial_{t,\rho}^{\ast}\varphi\rangle_{L_{2,\rho}(\R;H)}\\
 & =\langle Pf,\varphi\rangle_{H_{\rho}^{-1}\times H_{\rho}^{1}}+\e^{-2\rho t}\left(\partial_{t,\rho}^{-1}f\right)(t+)\varphi(t)\\
 & =0,
\end{align*}
which gives $\spt f\subseteq\R_{\leq t}.$\qedhere 

\end{enumerate}
\end{proof}

\section{Admissible histories for evolutionary equations\label{sec:Admissible-histories-for}}

In this section we study evolutionary problems of the following form
\begin{align}
\left(\partial_{t,\rho}M(\partial_{t,\rho})+A\right)u & =0\quad\text{on }\R_{>0},\nonumber \\
u & =g\quad\text{on }\R_{<0},\label{eq:IVP}
\end{align}
where $M$ and $A$ are as in \prettyref{thm:well_posed} and $g$
is a given function on $\R_{<0}$. The first goal is to rewrite this
`Initial value problem' into a proper evolutionary equations as
it is introduced in \prettyref{sec:Evolutionary-Problems}. For doing
so, we start with some heuristics to motivate the definition which
will be made below. In particular, for the moment we will not care
about domains of operators. \\
We will now write \prettyref{eq:IVP} as an evolutionary equation
for the unknown $v\coloneqq u|_{\R_{\geq0}}$, which is the part of
$u$ to be determined. For doing so, we first assume that $u\in H_{\rho}^{1}(\R;H)$
for some $\rho>0,$ which means that $v+g\in H_{\rho}^{1}(\R;H).$
We interpret the first line of \prettyref{eq:IVP} as 
\[
P_{0}\left(\partial_{t,\rho}M(\partial_{t,\rho})+A\right)u=0,
\]
where $P_{0}$ is the cut-off operator introduced in \prettyref{sec:Cut-off-operators}.
The latter gives 
\begin{align*}
0 & =P_{0}\left(\partial_{t,\rho}M(\partial_{t,\rho})+A\right)u\\
 & =P_{0}\left(\partial_{t,\rho}M(\partial_{t,\rho})+A\right)v+P_{0}\left(\partial_{t,\rho}M(\partial_{t,\rho})+A\right)g\\
 & =\partial_{t,\rho}P_{0}M(\partial_{t,\rho})v+AP_{0}v-\left(M(\partial_{t,\rho})v\right)(0+)\delta_{0}+P_{0}\partial_{t,\rho}M(\partial_{t,\rho})g+AP_{0}g\\
 & =\partial_{t,\rho}P_{0}M(\partial_{t,\rho})v+Av+P_{0}\partial_{t,\rho}M(\partial_{t,\rho})g-\left(M(\partial_{t,\rho})v\right)(0+)\delta_{0}.
\end{align*}
Since $v$ is supported on $\R_{\geq0}$ by assumption and $M(\partial_{t,\rho})$
is causal by \prettyref{prop:material law}, we infer that $M(\partial_{t,\rho})v$
is also supported on $\R_{\geq0}$ and so, $P_{0}M(\partial_{t,\rho})v=M(\partial_{t,\rho})v.$
Hence, we arrive at an evolutionary problem for $v$ of the form 
\[
\left(\partial_{t,\rho}M(\partial_{t,\rho})+A\right)v=\left(M(\partial_{t,\rho})v\right)(0+)\delta_{0}-P_{0}\partial_{t,\rho}M(\partial_{t,\rho})g.
\]

Since $u=v+g\in H_{\rho}^{1}(\R;H)$ by assumption, we infer that
$u$ is continuous by \prettyref{prop:td} (e) and hence, the limits
$v(0+)$ and $g(0-)$ exist and coincide. Hence, $v-\chi_{\R_{\geq0}}g(0-)\in H_{\rho}^{1}(\R;H)$
and vanishes on $\R_{<0}.$ The latter gives
\begin{align*}
\left(M(\partial_{t,\rho})v\right)(0+) & =\left(M(\partial_{t,\rho})(v-\chi_{\R_{\geq0}}g(0-))\right)(0+)+\left(M(\partial_{t,\rho})\chi_{\R_{\geq0}}g(0-)\right)(0+)\\
 & =\left(M(\partial_{t,\rho})\chi_{\R_{\geq0}}g(0-)\right)(0+),
\end{align*}

where in the last equality we have used that $M(\partial_{t,\rho})(v-\chi_{\R_{\geq0}}g(0-))\in H_{\rho}^{1}(\R;H)$,
hence it is continuous, and vanishes on $\R_{\leq0}$ due to causality.
Summarising, we end up with the following problem for $v$
\begin{equation}
\left(\partial_{t,\rho}M(\partial_{t,\rho})+A\right)v=\left(M(\partial_{t,\rho})\chi_{\R_{\geq0}}g(0-)\right)(0+)\delta_{0}-P_{0}\partial_{t,\rho}M(\partial_{t,\rho})g.\label{eq:IVP_new}
\end{equation}

Now, to make sense of \prettyref{eq:IVP_new} we need to ensure that
the right hand side is well-defined. In particular, we need that $\left(M(\partial_{t,\rho})\chi_{\R_{\geq0}}g(0-)\right)(0+)$
exists. In order to ensure that, we introduce the following notion.
\begin{defn*}
Let $H$ be a Hilbert space, $\rho_{0}\geq0$ and $M:\C_{\Re>\rho_{0}}\to L(H)$
be analytic and bounded. We call $M$ \emph{regularising}, if for
all $x\in H,\rho>\rho_{0}$ the limit 
\[
\left(M(\partial_{t,\rho})\chi_{\R_{\geq0}}x\right)(0+)
\]
exists. Moreover, for $\rho>0$ we define the space 
\[
H_{\rho}^{1}(\R_{\leq0};H)\coloneqq\left\{ f|_{\R_{\leq0}}\,;\,f\in H_{\rho}^{1}(\R;H)\right\} .
\]
\end{defn*}
As it turns out, this assumption suffices to obtain a well-defined
expression on the right hand side of \prettyref{eq:IVP_new}.
\begin{prop}
Let $H$ be a Hilbert space, $\rho_{0}\geq0$ and $M:\C_{\Re>\rho_{0}}\to L(H)$
be analytic and bounded and assume that $M$ is regularising. Then
for each $g\in H_{\rho}^{1}(\R_{\leq0};H)$ with $\rho>\rho_{0}$
we have that 
\[
\partial_{t,\rho}M(\partial_{t,\rho})g\in\dom(P_{0}).
\]
\end{prop}

\begin{proof}
By assumption $g=f|_{\R_{\leq0}}$ for some $f\in H_{\rho}^{1}(\R;H).$
Hence, $g(0-)=f(0)$ exists and hence, an easy computation shows that
$g-\chi_{\R_{\leq0}}g(0-)\in H_{\rho}^{1}(\R;H)$. Hence, also $M(\partial_{t,\rho})\left(g-\chi_{\R_{\geq0}}g(0-)\right)\in H_{\rho}^{1}(\R;H)$
and thus, 
\[
\left(M(\partial_{t,\rho})g\right)(0+)=\left(M(\partial_{t,\rho})\left(g-\chi_{\R_{\geq0}}g(0-)\right)\right)(0+)+\left(M(\partial_{t,\rho})\chi_{\R_{\geq0}}g(0-)\right)(0+)
\]
exists and so, $\partial_{t,\rho}M(\partial_{t,\rho})g\in\dom(P_{0}).$ 
\end{proof}
We are now in the position to define the space of admissible history
functions $g$.
\begin{defn*}
Let $H$ be a Hilbert space, $\rho_{0}\geq0$ and $M:\C_{\Re>\rho_{0}}\to L(H)$
be analytic, bounded and regularising. Moreover, let $A:\dom(A)\subseteq H\to H$
be densely defined, closed and linear. For notational convenience,
we set 
\[
\Gamma_{\rho}:H_{\rho}^{1}(\R_{\leq0};H)\to H,\quad g\mapsto\left(M(\partial_{t,\rho})\chi_{\R_{\geq0}}g(0-)\right)(0+)
\]
and 
\[
K_{\rho}:H_{\rho}^{1}(\R_{\leq0};H)\to H_{\rho}^{-1}(\R;H),\quad g\mapsto P_{0}\partial_{t,\rho}M(\partial_{t,\rho})g
\]
for $\rho>\rho_{0}.$ Furthermore, we assume that the evolutionary
problem associated with $(M,A)$ is well-posed and define 
\[
\His_{\rho}\coloneqq\{g\in H_{\rho}^{1}(\R_{\leq0};H)\,;\,S_{\rho}\left(\Gamma_{\rho}g\delta_{0}-K_{\rho}g\right)+g\in H_{\rho}^{1}(\R;H)\}
\]
for each $\rho>s_{0}(M,A),$ the \emph{space of admissible histories}.
Here $S_{\rho}$ denotes the extension of the solution operator $(\partial_{t,\rho}M(\partial_{t,\rho})+A)^{-1}$
to $H_{\rho}^{-1}(\R;H)$ (cp. \prettyref{prop:extension_sol_op}).
Moreover, we set 
\[
\IV_{\rho}\coloneqq\left\{ g(0-)\,;\,g\in\His_{\rho}\right\} 
\]
the space of \emph{admissible initial values.}
\end{defn*}
\begin{rem}
\label{rem:Gamma}We have 
\[
\Gamma_{\rho}g=(M(\partial_{t,\rho})g)(0-)-\left(M(\partial_{t,\rho})g\right)(0+)
\]
for $g\in H_{\rho}^{1}(\R_{\leq0};H).$ Indeed, since $M(\partial_{t,\rho})$
is causal we infer 
\begin{align*}
(M(\partial_{t,\rho})g)(0-) & =(M(\partial_{t,\rho})(g+\chi_{\R_{\geq0}}g(0-)))(0-)\\
 & =(M(\partial_{t,\rho})(g+\chi_{\R_{\geq0}}g(0-)))(0+),
\end{align*}
since $g+\chi_{\R_{\geq0}}g(0-)\in H_{\rho}^{1}(\R;H).$ Thus, 
\[
(M(\partial_{t,\rho})g)(0-)-\left(M(\partial_{t,\rho})g\right)(0+)=\left(M(\partial_{t,\rho})\chi_{\R_{\geq0}}g(0-)\right)(0+)=\Gamma_{\rho}g.
\]
\end{rem}

We come back to the heuristic computation at the beginning of this
section and show, that for $g\in\His_{\rho}$ the computation can
be made rigorously.
\begin{prop}
Let $H$ be a Hilbert space, $\rho_{0}\geq0$ and $M:\C_{\Re>\rho_{0}}\to L(H)$
be analytic, bounded and regularising. Moreover, let $A:\dom(A)\subseteq H\to H$
be densely defined, closed and linear and assume that the evolutionary
problem associated with $(M,A)$ is well-posed. Let $\rho>s_{0}(M,A)$
and $g\in\His_{\rho}$. We set 
\[
v\coloneqq S_{\rho}\left(\Gamma_{\rho}g\delta_{0}-K_{\rho}g\right)
\]
and $u\coloneqq v+g.$ Then $\spt v\subseteq\R_{\geq0}$, $u\in H_{\rho}^{1}(\R;H)$
and satisfies \prettyref{eq:IVP}. 
\end{prop}

\begin{proof}
Note that by assumption $u=v+g\in H_{\rho}^{1}(\R;H)$ and thus, $v=u-g\in L_{2,\rho}(\R;H).$
We prove that $\spt v\subseteq\R_{\geq0}.$ For doing so, we compute
\begin{align*}
\partial_{t,\rho}^{-1}v & =\partial_{t,\rho}^{-1}S_{\rho}\left(\Gamma_{\rho}g\delta_{0}-K_{\rho}g\right)\\
 & =S_{\rho}\left(\partial_{t,\rho}^{-1}\Gamma_{\rho}g\delta_{0}-\partial_{t,\rho}^{-1}K_{\rho}g\right)\\
 & =S_{\rho}\left(\Gamma_{\rho}g\chi_{\R_{\geq0}}-\chi_{\R_{\geq0}}(\m)M(\partial_{t,\rho})g+\left(M(\partial_{t,\rho})g\right)(0+)\chi_{\R_{\geq0}}\right)
\end{align*}
and hence, $\spt\partial_{t,\rho}^{-1}v\subseteq\R_{\geq0}$ by causality
of $S_{\rho}.$ The latter implies $\spt v\subseteq\R_{\geq0}$. Thus,
we have $u=g$ on $\R_{<0}$ and we are left to show 
\[
\left(\partial_{t,\rho}M(\partial_{t,\rho})+A\right)u=0\quad\text{on }\R_{>0}.
\]

For doing so, let $\varphi\in C_{c}^{\infty}(\R_{>0};\dom(A^{\ast}))$.
We compute 
\begin{align*}
 & \langle\left(\partial_{t,\rho}M(\partial_{t,\rho})+A\right)u,\varphi\rangle_{L_{2,\rho}(\R;H^{-1}(|A^{\ast}|+1))\times L_{2,\rho}(\R;H^{1}(|A^{\ast}|+1))}\\
= & \langle u,\left(\partial_{t,\rho}M(\partial_{t,\rho})+A\right)^{\ast}\varphi\rangle_{L_{2,\rho}(\R;H)}\\
= & \langle v,\left(\partial_{t,\rho}M(\partial_{t,\rho})+A\right)^{\ast}\varphi\rangle_{L_{2,\rho}(\R;H)}+\langle g,\left(\partial_{t,\rho}M(\partial_{t,\rho})+A\right)^{\ast}\varphi\rangle_{L_{2,\rho}(\R;H)}\\
= & \langle\Gamma_{\rho}g\delta_{0}-K_{\rho}g,\varphi\rangle_{H_{\rho}^{-1}(\R;H)\times H_{\rho}^{1}(\R;H)}+\langle g,\left(\partial_{t,\rho}M(\partial_{t,\rho})\right)^{\ast}\varphi\rangle_{L_{2,\rho}(\R;H)},
\end{align*}
where in the last line we have used $\langle g,A^{\ast}\varphi\rangle=0,$
since $\spt g\subseteq\R_{\leq0}$. Moreover, we compute 
\begin{align*}
 & \langle\Gamma_{\rho}g\delta_{0}-K_{\rho}g,\varphi\rangle_{H_{\rho}^{-1}(\R;H)\times H_{\rho}^{1}(\R;H)}\\
 & =-\langle K_{\rho}g,\varphi\rangle_{H_{\rho}^{-1}(\R;H)\times H_{\rho}^{1}(\R;H)}\\
 & =-\langle\partial_{t,\rho}\chi_{\R_{\geq0}}(\m)M(\partial_{t,\rho})g-\left(M(\partial_{t,\rho})g\right)(0+)\delta_{0},\varphi\rangle_{H_{\rho}^{-1}(\R;H)\times H_{\rho}^{1}(\R;H)}\\
 & =-\langle M(\partial_{t,\rho})g,\partial_{t,\rho}^{\ast}\varphi\rangle_{L_{2,\rho}(\R;H)},
\end{align*}
where we have used two times that $\varphi(0)=0.$ Plugging this formula
in the above computation, we infer that 
\[
\langle\left(\partial_{t,\rho}M(\partial_{t,\rho})+A\right)u,\varphi\rangle_{L_{2,\rho}(\R;H^{-1}(|A^{\ast}|+1))\times L_{2,\rho}(\R;H^{1}(|A^{\ast}|+1))}=0,
\]
which shows the claim.
\end{proof}

\section{$C_{0}$-semigroups associated with evolutionary problems\label{sec:-semigroups-associated-with}}

Throughout this section, let $H$ be a Hilbert space, $\rho_{0}\geq0$
and $M:\C_{\Re>\rho_{0}}\to L(H)$ analytic, bounded and regularising.
Moreover, let $A:\dom(A)\subseteq H\to H$ be densely defined, closed
and linear such that the evolutionary problem associated with $(M,A)$
is well-posed. \\
In this section we aim for a $C_{0}$-semigroup associated with the
evolutionary problem for $(M,A)$ acting on a suitable subspace of
$\IV_{\rho}\times\His_{\rho}$ for $\rho>s_{0}(M,A).$ For doing so,
we first need to prove that $\His_{\rho}$ is left invariant by the
time evolution. The precise statement is as follows.
\begin{thm}
\label{thm:invariance}Let $\rho>s_{0}(M,A)$ and $g\in\His_{\rho}.$
Moreover, let $v\coloneqq S_{\rho}\left(\Gamma_{\rho}g\delta_{0}-K_{\rho}g\right)$
and $u\coloneqq v+g.$ For $t>0$ we set $h\coloneqq\chi_{\R_{\leq0}}(\m)u(t+\cdot)$
and $w\coloneqq\chi_{\R_{\geq0}}(\m)u(t+\cdot).$ Then $h\in\His_{\rho}$
and 
\[
w=S_{\rho}\left(\Gamma_{\rho}h\delta_{0}-K_{\rho}h\right).
\]

In particular, $w(0+)=h(0-)\in\IV_{\rho}.$ 
\end{thm}

\begin{proof}
We first note that
\[
\left(\partial_{t,\rho}M(\partial_{t,\rho})+A\right)\tau_{t}=\tau_{t}\left(\partial_{t,\rho}M(\partial_{t,\rho})+A\right),
\]
where $\tau_{t}u\coloneqq u(t+\cdot)$ for $u\in L_{2,\rho}(\R;H)$,
and hence, 
\[
\spt\left(\partial_{t,\rho}M(\partial_{t,\rho})+A\right)\tau_{t}u\subseteq\R_{\leq0}.
\]
The latter gives, employing the causality of $M(\partial_{t,\rho})$,
\begin{align*}
\left(\partial_{t,\rho}M(\partial_{t,\rho})+A\right)\tau_{t}u & =\chi_{\R_{\leq0}}(\m)\left(\partial_{t,\rho}M(\partial_{t,\rho})+A\right)\tau_{t}u\\
 & =\partial_{t,\rho}\chi_{\R_{\leq0}}(\m)M(\partial_{t,\rho})\tau_{t}u+\left(M(\partial_{t,\rho})\tau_{t}u\right)(0-)\delta_{0}+Ah\\
 & =\partial_{t,\rho}\chi_{\R_{\leq0}}(\m)M(\partial_{t,\rho})h+\left(M(\partial_{t,\rho})h\right)(0-)\delta_{0}+Ah\\
 & =Q_{0}\partial_{t,\rho}M(\partial_{t,\rho})h+Ah.
\end{align*}

The latter yields 
\begin{align*}
\left(\partial_{t,\rho}M(\partial_{t,\rho})+A\right)w & =\left(\partial_{t,\rho}M(\partial_{t,\rho})+A\right)\left(\tau_{t}u-h\right)\\
 & =Q_{0}\partial_{t,\rho}M(\partial_{t,\rho})h-\partial_{t,\rho}M(\partial_{t,\rho})h.
\end{align*}
Now, since $\partial_{t,\rho}M(\partial_{t,\rho})h\in\dom(P_{0})$
by causality of $M(\partial_{t,\rho}),$ we use \prettyref{prop:properties_cut_off}
(b) and \prettyref{rem:Gamma} to derive 
\begin{align*}
\left(\partial_{t,\rho}M(\partial_{t,\rho})+A\right)w & =-P_{0}\partial_{t,\rho}M(\partial_{t,\rho})h-\left(\left(M(\partial_{t,\rho})h\right)(0+)-\left(M(\partial_{t,\rho})h\right)(0-)\right)\delta_{0}\\
 & =\Gamma_{\rho}h\delta_{0}-K_{\rho}h,
\end{align*}
which yields the desired formula for $w$. Now $h\in\His_{\rho}$
follows, since by definition 
\[
S_{\rho}\left(\Gamma_{\rho}h\delta_{0}-K_{\rho}h\right)+h=w+h=\tau_{t}u\in H_{\rho}^{1}(\R;H).\tag*{\qedhere}
\]
\end{proof}
The latter theorem allows for defining a semigroup associated with
$(M,A).$
\begin{defn*}
Let $\rho>s_{0}(M,A)$ and set 
\[
D_{\rho}\coloneqq\{(g(0-),g)\,;\,g\in\His_{\rho}\}.
\]
For $g\in\His_{\rho}$ we set 
\[
v\coloneqq S_{\rho}\left(\Gamma_{\rho}g\delta_{0}-K_{\rho}g\right)
\]
and $u\coloneqq v+g.$ For $t\geq0$ we define 
\begin{align*}
T_{1}^{\rho}(t): & D_{\rho}\subseteq\IV_{\rho}\times\His_{\rho}\to\IV_{\rho},\quad(g(0-),g)\mapsto v(t+),\\
T_{2}^{\rho}(t): & D_{\rho}\subseteq\IV_{\rho}\times\His_{\rho}\to\His_{\rho},\quad(g(0-),g)\mapsto\chi_{\R_{\leq0}}(\m)\tau_{t}u
\end{align*}
and 
\[
T^{\rho}(t)\coloneqq(T_{1}^{\rho}(t),T_{2}^{\rho}(t)):D_{\rho}\subseteq\IV_{\rho}\times\His_{\rho}\to\IV_{\rho}\times\His_{\rho}.
\]
We call $(T^{\rho}(t))_{t\geq0}$ the \emph{semigroup associated with
$(M,A)$. }
\end{defn*}
First we show that $T^{\rho}$ defined above is indeed a strongly
continuous semigroup.
\begin{prop}
Let $\rho>s_{0}(M,A)$ and $T^{\rho}$ be the semigroup associated
with $(M,A).$ Then $T^{\rho}$ is a strongly continuous semigroup.
More precisely,
\[
T^{\rho}(t+s)=T^{\rho}(t)T^{\rho}(s)\quad(t,s\geq0)
\]
and 
\[
T^{\rho}(t)(g(0-),g)\to(g(0-),g)\quad(t\to0+)
\]
in $H\times L_{2,\rho}(\R;H)$ for each $g\in\His_{\rho}.$ 
\end{prop}

\begin{proof}
Let $g\in\His_{\rho}$ and $t,s\geq0.$ We set $v\coloneqq S_{\rho}\left(\Gamma_{\rho}g\delta_{0}-K_{\rho}g\right)$
and $u\coloneqq v+g.$ By \prettyref{thm:invariance} we have that
\[
\chi_{\R_{\geq0}}(\m)\tau_{s}u=S_{\rho}\left(\Gamma_{\rho}\left(\chi_{\R_{\leq0}}(\m)\tau_{s}u\right)\delta_{0}-K_{\rho}\left(\chi_{\R_{\leq0}}(\m)\tau_{s}u\right)\right).
\]
and thus, 
\begin{align*}
T^{\rho}(t)T^{\rho}(s)(g(0-),g) & =T^{\rho}(t)\left(u(s),\chi_{\R_{\leq0}}(\m)\tau_{s}u\right)\\
 & =(u(t+s),\chi_{\R_{\leq0}}(\m)\tau_{t}\tau_{s}u)\\
 & =T^{\rho}(t+s)(g(0-),g).
\end{align*}
Moreover, 
\begin{align*}
 & \|T^{\rho}(t)(g(0-),g)-(g(0-),g)\|_{H\times L_{2,\rho}(\R;H)}^{2}\\
 & =\|u(t)-g(0-)\|_{H}^{2}+\|\chi_{\R_{\leq0}}(\m)\tau_{t}u-g\|_{L_{2,\rho}(\R;H)}\\
 & =\|u(t)-u(0)\|_{H}^{2}+\|\chi_{\R_{\leq0}}(\m)(\tau_{t}u-u)\|_{L_{2,\rho}(\R;H)}\\
 & \leq\|u(t)-u(0)\|_{H}^{2}+\|\tau_{t}u-u\|_{L_{2,\rho}(\R;H)}\to0\quad(t\to0+),
\end{align*}
by the continuity of $u$ and the strong continuity of translation
in $L_{2,\rho}.$ 
\end{proof}
In the rest of this section we show a characterisation result, when
$T^{\rho}$ can be extended to a $C_{0}$-semigroup on the space 
\[
X_{\rho}^{\mu}\coloneqq\overline{D_{\rho}}^{H\times L_{2,\mu}(\R;H)}\subseteq H\times L_{2,\mu}(\R;H)
\]
for some $\mu\leq\rho.$ We first prove a result that is suffices
to consider the family $T_{1}^{\rho}$.
\begin{prop}
\label{prop:T1_suffices}Let $\rho>s_{0}(M,A)$ and $\mu\leq\rho.$
Assume that 
\[
T_{1}^{\rho}:D_{\rho}\subseteq X_{\rho}^{\mu}\to C_{\omega}(\R_{\geq0};H)
\]
is bounded for some $\omega\in\R.$ Then 
\[
T_{2}^{\rho}:D_{\rho}\subseteq X_{\rho}^{\mu}\to C_{\max\{\mu,\omega\}+\varepsilon}(\R_{\geq0};L_{2,\mu}(\R;H))
\]
is bounded for each $\varepsilon>0.$ 
\end{prop}

\begin{proof}
Let $\varepsilon>0$ and $g\in\His_{\rho}.$ We note that 
\[
\left(T_{2}^{\rho}(t)(g(0-),g)\right)(s)=\begin{cases}
g(t+s) & \text{ if }s<-t,\\
T_{1}^{\rho}(t+s)(g(0-),g) & \text{ if }-t\leq s\leq0
\end{cases}\quad(t\geq0,s\leq0).
\]
Hence, we may estimate for $\varepsilon>0$
\begin{align*}
\|T_{2}^{\rho}(t)(g(0-),g)\|_{L_{2,\mu}(\R;H)}^{2} & =\int_{-\infty}^{-t}\|g(t+s)\|^{2}\e^{-2\mu s}\d s+\int_{-t}^{0}\|T_{1}^{\rho}(t+s)(g(0-),g)\|^{2}\e^{-2\mu s}\d s\\
 & \leq\int_{-\infty}^{0}\|g(s)\|^{2}\e^{-2\mu s}\d s\;\e^{2\mu t}+M\|(g(0-),g)\|_{X_{\rho}^{\mu}}^{2}\int_{-t}^{0}\e^{2\omega(t+s)}\e^{-2\mu s}\d s\\
 & =\|g\|_{L_{2,\mu}(\R;H)}^{2}\e^{2\mu t}+M\|(g(0-),g)\|_{X_{\rho}^{\mu}}^{2}\e^{2\omega t}\frac{1}{2(\omega-\mu)}(1-\e^{-2(\omega-\mu)t})\\
 & \leq\|g\|_{L_{2,\mu}(\R;H)}^{2}\e^{2\mu t}+M\|(g(0-),g)\|_{X_{\rho}^{\mu}}^{2}t\e^{2\max\{\mu,\omega\}t}\\
 & \leq C\e^{2(\max\{\mu,\omega)+\varepsilon)t}\|(g(0-),g)\|_{X_{\rho}^{\mu}}^{2}
\end{align*}
for each $g\in\His_{\rho},$ where $M$ denotes the norm of $T_{1}^{\rho}$
and $C\coloneqq\max_{t\geq0}(1+Mt)\e^{-2\varepsilon t}.$
\end{proof}
In order to extend $T_{1}^{\rho}$ to $X_{\rho}^{\mu}$ we make use
of the Widder-Arendt-Theorem.
\begin{thm}[{Widder-Arendt, \cite{Arendt1987},\cite[Theorem 2.2.3]{ABHN_2011}}]
\label{thm:WA} Let $H$ be a Hilbert space and $r\in C^{\infty}(\R_{>0};H)$
such that 
\[
M\coloneqq\sup_{\lambda>0,k\in\N}\frac{\lambda^{k+1}}{k!}\|r^{(k)}(\lambda)\|<\infty.
\]
Then there is $f\in L_{\infty}(\R_{\geq0};H)$ such that $\|f\|_{\infty}=M$
and 
\[
r(\lambda)=\int_{0}^{\infty}\e^{-\lambda t}f(t)\d t\quad(\lambda>0).
\]
\end{thm}

\begin{rem}
The latter Theorem was first proved by Widder in the scalar-valued
case \cite{Widder1934} and then generalised by Arendt to the vector-valued
case in \cite{Arendt1987}. It is noteworthy that the latter Theorem
is also true in Banach spaces satisfying the Radon-Nikodym property
(see \cite[Chapter III]{DiestelUhl}) and, in fact, this property
of $X$ is equivalent to the validity of \prettyref{thm:WA}, see
\cite[Theorem 1.4]{Arendt1987}.
\end{rem}

We now identify the function $r$ mentioned in \prettyref{thm:WA}
within the presented framework.
\begin{prop}
\label{prop:function_r}Let $\rho>s_{0}(M,A)$ and $g\in\His_{\rho}.$
We set $v\coloneqq S_{\rho}\left(\Gamma_{\rho}g\delta_{0}-K_{\rho}g\right)\in L_{2,\rho}(\R;H)$
and 
\[
r_{g}(\lambda)\coloneqq\sqrt{2\pi}(\mathcal{L}_{\lambda}v)(0)\quad(\lambda>\rho).
\]
Then $r_{g}\in C^{\infty}(\R_{>\rho};H)$. Moreover, 
\[
r_{g}(\lambda)=(\lambda M(\lambda)+A)^{-1}\left(\left(M(\partial_{t,\rho})g\right)(0-)-\lambda\sqrt{2\pi}\mathcal{L}_{\lambda}(\chi_{\mathbb{R}_{\geq0}}(\m)M(\partial_{t,\rho})g)(0)\right)\quad(\lambda>\rho).
\]
\end{prop}

\begin{proof}
We note that 
\[
(\mathcal{L}_{\lambda}v)(0)=\frac{1}{\sqrt{2\pi}}\int_{0}^{\infty}\e^{-\lambda t}v(t)\d t\quad(\lambda>\rho)
\]
and hence, the regularity of $r_{g}$ follows. Moreover, 
\begin{align*}
\partial_{\lambda,t}^{-1}v & =\partial_{t,\rho}^{-1}v\\
 & =S_{\rho}\left(\partial_{t,\rho}^{-1}\Gamma_{\rho}g\delta_{0}-\partial_{t,\rho}^{-1}K_{\rho}g\right)\\
 & =S_{\rho}\left(\Gamma_{\rho}g\chi_{\R_{\geq0}}-\chi_{\R_{\geq0}}(\m)M(\partial_{t,\rho})g+(M(\partial_{t,\rho})g)(0+)\chi_{\R_{\geq0}}\right)\\
 & =S_{\lambda}\left(\Gamma_{\rho}g\chi_{\R_{\geq0}}-\chi_{\R_{\geq0}}(\m)M(\partial_{t,\rho})g+(M(\partial_{t,\rho})g)(0+)\chi_{\R_{\geq0}}\right),
\end{align*}
where we have used the independence of $\rho$ stated in \prettyref{thm:well_posed}.
Hence, 
\begin{align*}
r_{g}(\lambda) & =\sqrt{2\pi}(\mathcal{L}_{\lambda}v)(0)\\
 & =\lambda\sqrt{2\pi}(\mathcal{L}_{\lambda}\partial_{t,\lambda}^{-1}v)(0)\\
 & =\lambda\sqrt{2\pi}\left(\lambda M(\lambda)+A\right)^{-1}\left(\frac{1}{\lambda\sqrt{2\pi}}\Gamma_{\rho}g-\mathcal{L}_{\lambda}\left(\chi_{\R_{\geq0}}(\m)M(\partial_{t,\rho})g\right)(0)+\frac{1}{\lambda\sqrt{2\pi}}\left(M(\partial_{t,\rho})g\right)(0+)\right)\\
 & =\left(\lambda M(\lambda)+A\right)^{-1}\left(\Gamma_{\rho}g+\left(M(\partial_{t,\rho})g\right)(0+)-\lambda\sqrt{2\pi}\mathcal{L}_{\lambda}\left(\chi_{\R_{\geq0}}(\m)M(\partial_{t,\rho})g\right)(0)\right)\\
 & =\left(\lambda M(\lambda)+A\right)^{-1}\left(\left(M(\partial_{t,\rho})g\right)(0-)-\lambda\sqrt{2\pi}\mathcal{L}_{\lambda}\left(\chi_{\R_{\geq0}}(\m)M(\partial_{t,\rho})g\right)(0)\right)
\end{align*}
for each $\lambda>\rho,$ where we have used the formula for $\Gamma_{\rho}$
stated in \prettyref{rem:Gamma}. 
\end{proof}
With these preparations at hand, we can now state and prove the main
result of this article.
\begin{thm}
\label{thm:HY}Let $\rho>s_{0}(M,A)$ and $T^{\rho}$ be the semigroup
on $D_{\rho}$ associated with $(M,A).$ Moreover, for $g\in\His_{\rho}$
we set 
\[
r_{g}(\lambda)\coloneqq\left(\lambda M(\lambda)+A\right)^{-1}\left(\left(M(\partial_{t,\rho})g\right)(0-)-\lambda\sqrt{2\pi}\mathcal{L}_{\lambda}\left(\chi_{\R_{\geq0}}(\m)M(\partial_{t,\rho})g\right)(0)\right)\quad(\lambda>\rho).
\]
For $\mu\leq\rho$ the following statements are equivalent:

\begin{enumerate}[(i)]

\item $T^{\rho}$ can be extended to a $C_{0}$-semigroup on $X_{\rho}^{\mu}=\overline{D_{\rho}}^{H\times L_{2,\mu}(\R;H)}\subseteq H\times L_{2,\mu}(\R;H).$

\item There exists $M\geq1$ and $\omega\geq\rho$ such that 
\[
\frac{(\lambda-\omega)^{k+1}}{k!}\|r_{g}^{(k)}(\lambda)\|\leq M\left(\|g(0-)\|_{H}+\|g\|_{L_{2,\mu}(\R;H)}\right)
\]
for each $\lambda>\omega,k\in\N$ and $g\in\His_{\rho}.$

\end{enumerate}

In this case 
\begin{align*}
T_{1}^{\rho} & :X_{\rho}^{\mu}\to C_{\omega}(\R_{\geq0};H),\\
T_{2}^{\rho} & :X_{\rho}^{\mu}\to C_{\omega+\varepsilon}(\R_{\geq0};L_{2,\mu}(\R;H))
\end{align*}
are bounded for each $\varepsilon>0.$ 
\end{thm}

\begin{proof}
(i) $\Rightarrow$(ii): Since $T^{\rho}:X_{\rho}^{\mu}\to X_{\rho}^{\mu}$
is a $C_{0}$-semigroup, we find $M\geq1$ and $\omega\geq\rho$ such
that 
\[
\|T^{\rho}(t)\|\leq M\e^{\omega t}\quad(t\geq0).
\]
In particular, we infer that 
\[
\|T_{1}^{\rho}(t)(g(0-),g)\|\leq M\e^{\omega t}\|(g(0-),g)\|_{X_{\rho}^{\mu}}\quad(t\geq0,g\in\His_{\rho}).
\]
Since $r_{g}(\lambda)=\sqrt{2\pi}\mathcal{L}_{\lambda}\left(T_{1}^{\rho}(\cdot)(g(0-),g)\right)(0)$
for $\lambda>\omega$ by \prettyref{prop:function_r}, we infer that
\begin{align*}
\|r_{g}^{(k)}(\lambda)\| & =\left\Vert \int_{0}^{\infty}\e^{-\lambda t}(-t)^{k}T_{1}^{\rho}(t)(g(0-),g)\d t\right\Vert \\
 & \leq\int_{0}^{\infty}\e^{-\lambda t}t^{k}M\e^{\omega t}\d t\|(g(0-),g)\|_{X_{\rho}^{\mu}}\\
 & =M\frac{k!}{(\lambda-\omega)^{k+1}}\|(g(0-),g)\|_{X_{\rho}^{\mu}},
\end{align*}
which shows (ii).\\
(ii)$\Rightarrow$(i): Let $g\in\His_{\rho}$ and define $\tilde{r}:\R_{>0}\to H$
by $\tilde{r}(\lambda)=r_{g}(\lambda+\omega)$ for $\lambda>0.$ Then
$\tilde{r}$ satisfies the assumptions of \prettyref{thm:WA} and
hence, there is $f\in L_{\infty}(\R_{\geq0};H)$ with $\|f\|_{\infty}\le M\left(\|g(0-)\|_{H}+\|g\|_{L_{2,\mu}(\R;H)}\right)$
such that 
\[
r_{g}(\lambda+\omega)=\int_{0}^{\infty}\e^{-\lambda t}f(t)\d t=\int_{0}^{\infty}\e^{-(\lambda+\omega)t}\e^{\omega t}f(t)\d t
\]
for each $\lambda>0.$ In particular, setting $v\coloneqq T_{1}^{\rho}(\cdot)(g(0-),g)$
we obtain 
\[
\int_{0}^{\infty}\e^{-\lambda t}v(t)\d t=r_{g}(\lambda)=\int_{0}^{\infty}\e^{-\lambda t}\e^{\omega t}f(t)\d t\quad(\lambda>\omega)
\]
and by analytic extension it follows that 
\[
\mathcal{L}_{\lambda}v=\mathcal{L}_{\lambda}(\e^{\omega\cdot}f)\quad(\lambda>\omega).
\]
Thus, $v=\e^{\omega\cdot}f$ and hence, 
\[
\|v(t)\|=\e^{\omega t}\|f(t)\|\leq M\e^{\omega t}\left(\|g(0-)\|_{H}+\|g\|_{L_{2,\mu}(\R;H)}\right).
\]
Thus, since $v$ is continuous on $\R_{\geq0}$, we derive that 
\[
T_{\rho}^{1}:D_{\rho}\subseteq X_{\rho}^{\mu}\to C_{\omega}(\R_{\geq0};H)
\]
is bounded and can therefore be extended to a $C_{0}$-semigroup on
$X_{\rho}^{\mu}$. Then, by \prettyref{prop:T1_suffices} we obtain
that 
\[
T_{\rho}^{2}:D_{\rho}\subseteq X_{\rho}^{\mu}\to C_{\omega+\varepsilon}(\R_{\geq0};L_{2,\mu}(\R;H))
\]
is also bounded for each $\varepsilon>0$ and hence, (i) follows.
\end{proof}

\section{Applications}

\subsection{Differential-algebraic equations and classical Cauchy problems}

In this section we consider initial value problems of the form 
\begin{align*}
\left(\partial_{t,\rho}E+A\right)u & =0\quad\text{on }\R_{>0},\\
u & =g\quad\text{on }\R_{\leq0},
\end{align*}
for a bounded operator $E\in L(H)$, $H$ a Hilbert space, and a densely
defined linear and closed operator $A:\dom(A)\subseteq H\to H.$ We
note that this corresponds to the abstract initial value problem \prettyref{eq:IVP}
with 
\[
M(z)\coloneqq E\quad(z\in\C).
\]
We assume that the evolutionary problem is well-posed, that is we
assume that there is $\rho_{1}\in\R_{\geq0}$ such that $zE+A$ is
boundedly invertible for each $z\in\C_{\Re\geq\rho_{1}}$ and 
\[
\sup_{z\in\C_{\Re\geq\rho_{1}}}\|(zE+A)^{-1}\|<\infty.
\]
We again denote the infimum over all such $\rho_{1}\in\R_{\geq0}$
by $s_{0}(E,A).$ 
\begin{lem}
\label{lem:His_DAE}The function $M$ given by $M(z)\coloneqq E$
for $z\in\C$ is regularising. Moreover, $\Gamma_{\rho}g=Eg(0-)$
and $K_{\rho}g=0$ for each $g\in H_{\rho}^{1}(\R_{\leq0};H)$ and
$\rho\in\R_{>0}.$ In particular, for $\rho>s_{0}(E,A)$ we have that
\[
\IV_{\rho}=\{x\in H\,;\,S_{\rho}(\delta Ex)-\chi_{\R_{\geq0}}x\in H_{\rho}^{1}(\R;H)\}
\]
and 
\[
\His_{\rho}=\{g\in H_{\rho}^{1}(\R_{\leq0};H)\,;\,g(0-)\in\IV_{\rho}\}.
\]

Moreover, $X_{\rho}^{\mu}=\overline{\IV_{\rho}}\times L_{2,\mu}(\R_{\leq0};H)$
for each $\mu\leq\rho.$
\end{lem}

\begin{proof}
For $x\in H$, $\rho>0$ we have 
\[
M(\partial_{t,\rho})\chi_{\R_{\geq0}}x=\chi_{\R_{\geq0}}Ex
\]
and thus, $M$ is regularising with $\Gamma_{\rho}g=Eg(0-)$ for each
$g\in H_{\rho}^{1}(\R_{\leq0};H).$ Moreover, we have 
\begin{align*}
K_{\rho}g & =P_{0}\partial_{t,\rho}Eg\\
 & =\partial_{t,\rho}\chi_{\R_{\geq0}}(\m)Eg-\delta_{0}(Eg)(0+)\\
 & =0.
\end{align*}
Hence, for $g\in H_{\rho}^{1}(\R_{\leq0};H)$, $\rho>s_{0}(E,A)$,
we have 
\begin{align*}
g\in\His_{\rho} & \Leftrightarrow\:S_{\rho}(\delta_{0}Eg(0-))+g\in H_{\rho}^{1}(\R;H)\\
 & \Leftrightarrow\:S_{\rho}(\delta_{0}Eg(0-))-\chi_{\R_{\geq0}}g(0-)+g+\chi_{\R_{\geq0}}g(0-)\in H_{\rho}^{1}(\R;H)\\
 & \Leftrightarrow\:S_{\rho}(\delta_{0}Eg(0-))-\chi_{\R_{\geq0}}g(0-)\in H_{\rho}^{1}(\R;H)
\end{align*}
which proves the asserted equalities for $\His_{\rho}$ and $\IV_{\rho}$.\\
Finally, let $x\in\overline{\IV_{\rho}}$ and $g\in L_{2,\mu}(\R_{\leq0};H)$
for some $\mu\leq\rho$ with $\rho>s_{0}(E,A).$ Then we find a sequence
$(x_{n})_{n\in\N}$ in $\IV_{\rho}$ and a sequence $(\varphi_{n})_{n\in\N}$
in $C_{c}^{\infty}(\R_{<0};H)$ such that $x_{n}\to x$ and $\varphi_{n}\to g$
in $H$ and $L_{2,\mu}(\R_{\leq0};H)$, respectively. Moreover, we
set 
\[
\psi_{n}(t)\coloneqq\begin{cases}
\left(nt+1\right)x_{n} & \text{ if }-\frac{1}{n}\leq t\leq0,\\
0 & \text{ else}
\end{cases}\quad(t\in\R_{\leq0},n\in\N)
\]
and obtain a sequence $(\psi_{n})_{n\in\N}$ in $H_{\rho}^{1}(\R_{\leq0};H)$
with $\psi_{n}(0-)=x_{n}$ for $n\in\N$ and $\psi_{n}\to0$ as $n\to\infty$
in $L_{2,\mu}(\R_{\leq0};H).$ Consequently, setting $g_{n}\coloneqq\psi_{n}+\varphi_{n}\in H_{\rho}^{1}(\R_{\leq0};H)$
for $n\in\N$ we obtain a sequence $(x_{n},g_{n})_{n\in\N}$ in $D_{\rho}$
with $(x_{n},g_{n})\to(x,g)$ in $H\times L_{2,\mu}(\R;H)$ and thus,
$(x,g)\in X_{\rho}^{\mu}.$ Since the other inclusion holds obviously,
this proves the assertion.  
\end{proof}
We now inspect the space $\IV_{\rho}$ a bit closer. In particular,
we are able to determine its closure $\overline{\IV_{\rho}}$ and
a suitable dense subset of $\overline{\IV_{\rho}}$.
\begin{prop}
\label{prop:IV_DAE}We set 
\[
U\coloneqq\{x\in\dom(A)\,;\,\exists y\in\dom(A):\:Ax=Ey\}.
\]
Then $U\subseteq\IV_{\rho}$ and $\overline{U}=\overline{\IV_{\rho}}$
for each $\rho>s_{0}(E,A)$. In particular, $\overline{\IV_{\rho}}$
does not depend on the particular choice of $\rho>s_{0}(E,A)$. 
\end{prop}

\begin{proof}
Let $\rho>s_{0}(E,A)$, $x\in U$ and $y\in\dom(A)$ with $Ax=Ey.$
Then we compute 
\begin{align*}
S_{\rho}\left(\delta Ex\right)-\chi_{\R_{\geq0}}x & =\left(\partial_{t,\rho}E+A\right)^{-1}(\delta Ex-\delta Ex-\chi_{\R_{\geq0}}Ax)\\
 & =-(\partial_{t,\rho}E+A)^{-1}(\chi_{\R_{\geq0}}Ey)\\
 & =-(\partial_{t,\rho}E+A)^{-1}(\partial_{t,\rho}E\partial_{t,\rho}^{-1}\chi_{\R_{\geq0}}y)\\
 & =-\partial_{t,\rho}^{-1}\chi_{\R_{\geq0}}y+(\partial_{t,\rho}E+A)^{-1}(\partial_{t,\rho}^{-1}\chi_{\R_{\geq0}}Ay)\in H_{\rho}^{1}(\R;H),
\end{align*}
which shows hat $x\in\IV_{\rho}$ by \prettyref{lem:His_DAE}. For
showing the remaining assertion, we prove that $\IV_{\rho}\subseteq\overline{U}.$
For doing so, let $x\in\IV_{\rho}$ and set $v\coloneqq S_{\rho}(\delta Ex).$
Then 
\begin{align*}
\partial_{t,\rho}E(v-\chi_{\R_{\geq0}}x) & =(\partial_{t,\rho}E+A)v-\delta Ex-Av\\
 & =-Av,
\end{align*}
and since the left-hand side belongs to $L_{2,\rho}(\R;H)$ we infer
that $v\in L_{2,\rho}(\R;\dom(A)).$ Hence, $\partial_{t,\rho}^{-1}v\in H_{\rho}^{1}(\R;\dom(A))\hookrightarrow C_{\rho}(\R;\dom(A))$
and so $\int_{0}^{t}v(s)\d s=\left(\partial_{t,\rho}^{-1}v\right)(t)\in\dom(A)$
for each $t\geq0$ and 
\[
A\int_{0}^{t}v(s)\d s=Ev(t)-Ex\quad(t\geq0).
\]
Consequently, 
\[
\int_{0}^{t}v(s)\d s\in A^{-1}[\ran(E)]\quad(t\geq0)
\]
and since $v$ is continuous on $\R_{\geq0}$ and hence, $\frac{1}{t}\int_{0}^{t}v(s)\d s\to v(0+)=x$
as $t\to0$, it suffices to prove $A^{-1}[\ran(E)]\subseteq\overline{U}$.
For doing so, let $y\in A^{-1}[\ran(E)]$, i.e., $y\in\dom(A)$ and
$Ay=Ez$ for some $z\in H.$ We choose a sequence $(z_{n})_{n\in\N}$
in $\dom(A)$ with $z_{n}\to z$ as $n\to\infty$ and define 
\[
y_{n}\coloneqq\left(\lambda E+A\right)^{-1}(\lambda Ey+Ez_{n})\quad(n\in\N),
\]
where $\lambda>s_{0}(E,A)$ is fixed. Then $y_{n}\in U,$ since 
\[
Ay_{n}=A\left(\lambda E+A\right)^{-1}(\lambda Ey+Ez_{n})=E\left(\lambda E+A\right)^{-1}(\lambda Ay+Az_{n})\in E[\dom(A)]
\]
and since $Ez_{n}\to Ez=Ay,$ we infer that $y_{n}\to y$ and hence,
$y\in\overline{U}$. 
\end{proof}
\begin{thm}
\label{thm:HY_DAE}Let $M(z)\coloneqq E$ for $z\in\C$, $\rho>s_{0}(E,A)$
and let $T^{\rho}:D_{\rho}\subseteq\IV_{\rho}\times\His_{\rho}\to\IV_{\rho}\times\His_{\rho}$
denote the semigroup associated with $(M,A).$ Moreover, for $x\in H$
we define 
\[
f_{x}(t)\coloneqq\begin{cases}
(t+1)x & \text{ if }t\in[-1,0],\\
0 & \text{ else}
\end{cases}\quad(t\in\R_{\leq0}).
\]
Then the following statements are equivalent:

\begin{enumerate}[(i)]

\item $T^{\rho}$ extends to a $C_{0}$-semigroup in $\overline{\IV_{\rho}}\times L_{2,\mu}(\R_{\leq0};H)$
for some $\mu\leq\rho.$

\item There exists $M\geq1$ and $\omega\geq\rho$ such that 
\begin{equation}
\|\left((\lambda E+A)^{-1}E\right)^{n}\|\leq\frac{M}{(\lambda-\omega)^{n}}\quad(\lambda>\omega,n\in\N).\label{eq:HY_condition}
\end{equation}

\item $T^{\rho}$ extends to a $C_{0}$-semigroup in $\overline{\IV_{\rho}}\times L_{2,\mu}(\R_{\leq0};H)$
for each $\mu\leq\rho.$

\item The family of functions 
\[
S^{\rho}(t):\IV_{\rho}\subseteq\overline{\IV_{\rho}}\to\overline{\IV_{\rho}},\quad x\mapsto T_{1}^{\rho}(t)(x,f_{x})
\]
for $t\geq0$ extends to a $C_{0}$-semigroup on $\overline{\IV_{\rho}}.$ 

\end{enumerate}

In the latter case, $S^{\rho}(t)x=T^{\rho}(t)(x,0)$ for each $x\in\overline{\IV_{\rho}}$
and $t\geq0$. 
\end{thm}

\begin{proof}
We first compute the function $r_{g}$ for $g\in\His_{\rho}$ as it
was defined in \prettyref{thm:HY}. We have that 
\[
\left(M(\partial_{t,\rho})g\right)(0-)-\lambda\sqrt{2\pi}\mathcal{L}_{\lambda}\left(\chi_{\R_{\geq0}}(\m)M(\partial_{t,\rho})g\right)(0)=Eg(0-)\quad(\lambda>\rho)
\]
and hence, 
\[
r_{g}(\lambda)=(\lambda E+A)^{-1}Eg(0-)\quad(\lambda>\rho).
\]
Consequently, 
\[
r_{g}^{(k)}(\lambda)=(-1)^{k}k!\left(\left(\lambda E+A\right)^{-1}E\right)^{k+1}g(0-)\quad(k\in\N_{0},\lambda>\rho).
\]

(i) $\Rightarrow$ (ii): By \prettyref{thm:HY} (note that $X_{\rho}^{\mu}=\overline{\IV_{\rho}}\times L_{2,\mu}(\R_{\leq0};H)$
by \prettyref{lem:His_DAE}) we know that there exists $M\geq1$ and
$\omega\geq\rho$ such that 
\[
\frac{(\lambda-\omega)^{k+1}}{k!}\|r_{g}^{(k)}(\lambda)\|\leq M\left(\|g(0-)\|_{H}+\|g\|_{L_{2,\mu}(\R;H)}\right)
\]
for each $\lambda>\omega,k\in\N$ and $g\in\His_{\rho}.$ Choosing
now $x\in\overline{\IV_{\rho}}$ we infer that 
\begin{align*}
\|\left((\lambda E+A)^{-1}E\right)^{n}x\| & =\frac{1}{(n-1)!}\|r_{f_{x}(k\cdot)}^{(n-1)}(\lambda)\|\\
 & \leq\frac{M}{(\lambda-\omega)^{n}}\left(\|x\|_{H}+\|f_{x}(k\cdot)\|_{L_{2,\mu}(\R;H)}\right)
\end{align*}
for each $\lambda>\omega$, $n,k\in\N$. Since $f_{x}(k\cdot)\to0$
as $k\to\infty$, we infer that 
\[
\|\left((\lambda E+A)^{-1}E\right)^{n}\|\leq\frac{M}{(\lambda-\omega)^{n}}\quad(\lambda>\omega,n\in\N).
\]
(ii) $\Rightarrow$ (iii): Let $\mu\leq\rho$. By assumption, there
exists $M\geq1,\omega\geq\rho$ such that 
\begin{align*}
\frac{(\lambda-\omega)^{k+1}}{k!}\|r_{g}^{(k)}(\lambda)\| & =(\lambda-\omega)^{k+1}\|\left((\lambda E+A)^{-1}E\right)^{k+1}g(0-)\|\\
 & \leq M\|g(0-)\|_{H}\\
 & \leq M\left(\|g(0-)\|+\|g\|_{L_{2,\mu}(\R;H)}\right)
\end{align*}
for each $\lambda>\omega,k\in\N_{0}$ and $g\in\His_{\rho}$ and hence,
the assertion follows from \prettyref{thm:HY} and \prettyref{lem:His_DAE}. 

(iii) $\Rightarrow$ (iv): Since $T^{\rho}$ extends to a $C_{0}$-semigroup
on $\overline{\IV_{\rho}}\times L_{2}(\R_{\leq0};H)$ and since 
\[
\|f_{x}\|_{L_{2}(\R;H)}\leq\|x\|_{H}\quad(x\in H),
\]
we infer that there is $M\geq1$ and $\omega\in\R$ such that 
\[
\|S^{\rho}(t)x\|\leq2M\e^{\omega t}\|x\|\quad(x\in\IV_{\rho})
\]
 and thus, $(S^{\rho}(t))_{t\geq0}$ extends to a $C_{0}$-semigroup
on $\overline{\IV_{\rho}}$. Moreover, since 
\begin{align*}
S^{\rho}(t)x & =T_{1}^{\rho}(t)(x,f_{x})\\
 & =\left((\partial_{t,\rho}E+A)^{-1}(\delta_{0}Ex)\right)(t)\\
 & =T_{1}^{\rho}(t)(x,0)
\end{align*}
for each $t\geq0,x\in\overline{\IV_{\rho}}$, we obtain the at the
end asserted formula .\\
(iv) $\Rightarrow$ (i): By assumption, there is $M\geq1,\omega\in\R$
such that 
\[
\|T_{1}^{\rho}(t)(x,f_{x})\|\leq M\e^{\omega t}\|x\|\quad(x\in\IV_{\rho},t\geq0).
\]
Moreover, since 
\[
T_{1}^{\rho}(t)(x,g)=T_{1}^{\rho}(t)(x,f_{x})\quad\left((x,g)\in D_{\rho}\right),
\]
we infer that 
\[
T_{1}^{\rho}:D_{\rho}\subseteq\overline{\IV_{\rho}}\times L_{2,\mu}(\R_{\leq0};H)\to C_{\omega}(\R_{\geq0};H)
\]
 is continuous and hence, the assertion follows by \prettyref{prop:T1_suffices}.
\end{proof}
\begin{rem}
We remark that in the case of classical Cauchy problems, i.e. $E=1$,
condition \prettyref{eq:HY_condition} is nothing but the classical
Hille-Yosida condition for generators of $C_{0}$-semigroups (see
e.g. \cite[Chapter II, Theorem 3.8]{engel2000one}). Note that in
this case, $U=\dom(A^{2})$ in \prettyref{prop:IV_DAE} and hence,
$\overline{\IV_{\rho}}=\overline{U}=H.$ 
\end{rem}

\subsection{A hyperbolic delay equation}

As a slight generalisation of \cite[Example 3.17]{Batkai_2005} we
consider a concrete delay equation of the form 
\begin{align}
\partial_{t,\rho}^{2}u-\dive k\grad u-\sum_{i=1}^{n}c_{i}\tau_{-h_{i}}\partial_{i}u-c_{0}\tau_{-h_{0}}\partial_{t,\rho}u & =0\quad\text{on }\R_{>0},\nonumber \\
u & =g\quad\text{on }\R_{<0}.\label{eq:delay_hyper}
\end{align}
Here, $u$ attains values in $L_{2}(\Omega)$ for some open set $\Omega\subseteq\R^{n}$
as underlying domain, $h_{0},\ldots,h_{n}>0$ are given real numbers
and $k,c_{0},\ldots,c_{n}$ are bounded operators on $L_{2}(\Omega)^{n}$
and $L_{2}(\Omega),$ respectively. The operators $\grad$ and $\dive$
denote the usual gradient and divergence with respect to the spatial
variables and will be introduced rigorously later. It is our first
goal to rewrite this equation as a suitable evolutionary problem.
For doing so, we need the following definition.
\begin{defn*}
Let $c_{0},\ldots,c_{n}\in L(L_{2}(\Omega))$ and $k\in L(L_{2}(\Omega)^{n})$
selfadjoint such that $k\geq d$ for some $d\in\R_{>0}$. We define
the function $M_{1}:\C\to L(L_{2}(\Omega)\times L_{2}(\Omega)^{n};L_{2}(\Omega))$
by 
\[
M_{1}(z)q\coloneqq c_{0}\e^{-h_{0}z}q_{0}-\sum_{i=1}^{n}c_{i}k^{-1}e^{-h_{i}z}q_{i}\quad(z\in\C,q\in L_{2}(\Omega)\times L_{2}(\Omega)^{n}).
\]
Furthermore, we define $M:\C\setminus\{0\}\to L(L_{2}(\Omega)\times L_{2}(\Omega)^{n})$
by 
\[
M(z)\left(\begin{array}{c}
v\\
q
\end{array}\right)\coloneqq\left(\begin{array}{c}
v+z^{-1}M_{1}(z)q\\
k^{-1}q
\end{array}\right).
\]
\end{defn*}
\begin{rem}
Since $\left(\mathcal{L}_{\rho}\tau_{h}u\right)(t)=\e^{(\i t+\rho)h}\left(\mathcal{L}_{\rho}u\right)(t)$
for each $u\in L_{2,\rho}(\R;H)$ and $t,h\in\R,$ we have that 
\[
M_{1}(\partial_{t,\rho})q=c_{0}\tau_{-h_{0}}q_{0}-\sum_{i=1}^{n}c_{i}k^{-1}\tau_{-h_{i}}q_{i}
\]
for each $q\in L_{2,\rho}(\R;L_{2}(\Omega)^{n}).$ 
\end{rem}

Obviously, the so defined function $M$ is analytic and if we restrict
it to some open half plane $\C_{\Re>\rho_{0}}$ with $\rho_{0}>0$,
it is bounded. Thus, we may consider the operator $M(\partial_{t,\rho})$
for some $\rho>0$. 
\begin{lem}
The function $M$ is regularising.
\end{lem}

\begin{proof}
We need to prove that $\left(M(\partial_{t,\rho})\chi_{\R_{\geq0}}x\right)(0+)$
exists for all $x=(\check{x},\hat{x})\in L_{2}(\Omega)\times L_{2}(\Omega)^{n}$
and $\rho>0$. We have that 
\[
M(\partial_{t,\rho})\chi_{\R_{\geq0}}x=\left(\begin{array}{c}
\chi_{\R_{\geq0}}\check{x}+\partial_{t,\rho}^{-1}M_{1}(\partial_{t,\rho})\chi_{\R_{\geq0}}x\\
k^{-1}\hat{x}
\end{array}\right)
\]
and since $M_{1}(\partial_{t,\rho})$ is causal, we infer that $\partial_{t,\rho}^{-1}M_{1}(\partial_{t,\rho})\chi_{\R_{\geq0}}x\in H_{\rho}^{1}(\R;L_{2}(\Omega))$
is supported on $\R_{>0}$ and hence, $\left(\partial_{t,\rho}^{-1}M_{1}(\partial_{t,\rho})\chi_{\R_{\geq0}}x\right)(0+)=0.$
Thus, $M$ is regularising.
\end{proof}
We now rewrite \prettyref{eq:delay_hyper} as an evolutionary equation.
We introduce $v\coloneqq\partial_{t,\rho}u$ and $q\coloneqq k\grad u$
as new unknowns, and rewrite \prettyref{eq:delay_hyper} as 
\begin{equation}
\left(\partial_{t,\rho}M(\partial_{t,\rho})+\left(\begin{array}{cc}
0 & \dive\\
\grad & 0
\end{array}\right)\right)\left(\begin{array}{c}
v\\
q
\end{array}\right)=0\quad\text{on }\R_{>0}.\label{eq:delay_hyper_evo}
\end{equation}

Of course \prettyref{eq:delay_hyper} has to be completed by suitable
boundary conditions. This will be done by introducing the differential
operators $\dive$ and $\grad$ in a suitable way.
\begin{defn*}
We define $\grad_{0}:\dom(\grad_{0})\subseteq L_{2}(\Omega)\to L_{2}(\Omega)^{n}$
as the closure of the operator 
\[
C_{c}^{\infty}(\Omega)\subseteq L_{2}(\Omega)\to L_{2}(\Omega)^{n},\:\varphi\mapsto\left(\partial_{j}\varphi\right)_{j\in\{1,\ldots,n\}}
\]
and similarly $\dive_{0}:\dom(\dive_{0})\subseteq L_{2}(\Omega)^{n}\to L_{2}(\Omega)$
as the closure of 
\[
C_{c}^{\infty}(\Omega)^{n}\subseteq L_{2}(\Omega)^{n}\to L_{2}(\Omega),\;(\varphi_{j})_{j\in\{1,\ldots,n\}}\mapsto\sum_{j=1}^{n}\partial_{j}\varphi_{j}.
\]
Moreover, we set 
\begin{align*}
\grad & \coloneqq-(\dive_{0})^{\ast}\\
\dive & \coloneqq-(\grad_{0})^{\ast}.
\end{align*}
\end{defn*}
\begin{rem}
We note that $\dom(\grad_{0})$ coincides with the classical Sobolev
space $H_{0}^{1}(\Omega)$ of weakly differentiable $L_{2}$-functions
with vanishing Dirichlet trace. Moreover, $\dom(\grad)$ is nothing
but the Sobolev space $H^{1}(\Omega)$. Thus, by Green's formula,
elements in $\dom(\dive_{0})$ may be interpreted as elements in $L_{2}(\Omega)^{n}$
with distributional divergence also lying in $L_{2}(\Omega)$ and
whose normal trace vanishes, while elements in $\dom(\dive)$ are
just $L_{2}(\Omega)$ vector fields with $L_{2}(\Omega)$-divergence.
Note however, that these definitions are meaningful for arbitrary
open sets $\Omega$ and do not require any boundary regularity.
\end{rem}

Thus, by replacing $\dive$ by $\dive_{0}$ or $\grad$ by $\grad_{0}$
in \prettyref{eq:delay_hyper_evo}, we can model homogeneous Neumann-
or Dirichlet conditions, respectively. 
\begin{lem}
We set 
\[
A_{N}\coloneqq\left(\begin{array}{cc}
0 & \dive_{0}\\
\grad & 0
\end{array}\right):\dom(\grad)\times\dom(\dive_{0})\subseteq L_{2}(\Omega)\times L_{2}(\Omega)^{n}\to L_{2}(\Omega)\times L_{2}(\Omega)^{n}
\]
and 
\[
A_{D}\coloneqq\left(\begin{array}{cc}
0 & \dive\\
\grad_{0} & 0
\end{array}\right):\dom(\grad_{0})\times\dom(\dive)\subseteq L_{2}(\Omega)\times L_{2}(\Omega)^{n}\to L_{2}(\Omega)\times L_{2}(\Omega)^{n}.
\]
Then both operators are skew-selfadjoint, i.e. $A_{N}^{\ast}=-A_{N}$
and $A_{D}^{\ast}=-A_{D}.$ 
\end{lem}

\begin{proof}
The claim follows immediately by the definitions of the differential
operators.
\end{proof}
We now prove that the evolutionary problems associated with $(M,A_{D/N})$
are well-posed. 
\begin{prop}
\label{prop:well_posed_delay}Let $c_{0},\ldots,c_{n}\in L(L_{2}(\Omega))$
and $k\in L(L_{2}(\Omega)^{n})$ selfadjoint such that $k\geq d$
for some $d\in\R_{>0}$. Then the evolutionary problems associated
with $(M,A_{D/N})$ are well-posed.
\end{prop}

\begin{proof}
We first note that $k^{-1}$ is selfadjoint and satisfies $k^{-1}\geq\frac{1}{\|k\|}.$
Moreover, since $A_{D/N}$ is skew-selfadjoint, we infer that 
\[
\Re\langle A_{D/N}x,x\rangle=0\quad(x\in\dom(A_{D/N})).
\]
Hence, we may estimate for $x=(x_{1},x_{2})\in\dom(A_{D/N})$
\begin{align*}
\Re\langle(zM(z)+A_{D/N})x,x\rangle & =\Re\langle zM(z)x,x\rangle\\
 & =\Re\langle zx_{1},x_{1}\rangle+\Re\langle zk^{-1}x_{2},x_{2}\rangle+\Re\langle M_{1}(z)x,x_{1}\rangle\\
 & \geq\Re z\min\{1,\frac{1}{\|k\|}\}\|x\|^{2}-\|M_{1}(z)\|\|x\|^{2}.
\end{align*}
Moreover, we estimate 
\[
\|M_{1}(z)\|\leq\|c_{0}\|\e^{-h_{0}\Re z}+\sum_{i=1}^{n}\|c_{i}\|\|k^{-1}\|\e^{-h_{i}\Re z}
\]
and hence, we infer that $\|M_{1}(z)\|\to0$ as $\Re z\to\infty.$
Thus, we find $c>0$ and $\rho_{0}>0$ such that 
\[
\Re\langle(zM(z)+A_{D/N})x,x\rangle\geq c\|x\|^{2}\quad(z\in\C_{\Re\geq\rho_{0}}),
\]
which yields the well-posedness for the evolutionary problem associated
with $(M,A_{D/N}).$ 
\end{proof}
\begin{rem}
We note that the above proof also works for $m$-accretive operators
$A$ instead of $A_{D/N}.$ This allows for the treatment of more
general boundary conditions and we refer to \cite{Trostorff2013_bd_maxmon}
for a characterisation result about those boundary conditions (including
also nonlinear ones).
\end{rem}

Having these results at hand, we are now in the position to consider
the history space for \prettyref{eq:delay_hyper_evo}. From now on,
to avoid cluttered notation, we will simply write $A$ and note that
$A$ can be replaced by $A_{N}$ and $A_{D}$, respectively.
\begin{prop}
\label{prop:His_delay}Let $\rho>s_{0}(M,A)$. Then 
\[
\Gamma_{\rho}g=\left(\begin{array}{cc}
1 & 0\\
0 & k^{-1}
\end{array}\right)g(0-),\quad K_{\rho}g=\chi_{\R_{\geq0}}(\m)\left(\begin{array}{c}
M_{1}(\partial_{t,\rho})\\
0
\end{array}\right)g
\]
for each $g\in H_{\rho}^{1}(\R_{\leq0};L_{2}(\Omega)\times L_{2}(\Omega)^{n}).$
Moreover, 
\begin{equation}
\left\{ g\in H_{\rho}^{1}(\R_{\leq0};\dom(A))\:;\:\forall j\in\{0,\ldots,n\}:g(-t_{j})=0,\:\left(\begin{array}{cc}
1 & 0\\
0 & k
\end{array}\right)Ag(0-)\in\dom(A)\right\} \subseteq\His_{\rho}\label{eq:subset_D_g}
\end{equation}
and consequently, 
\[
X_{\rho}^{\mu}=\left(L_{2}(\Omega)\times L_{2}(\Omega)^{n}\right)\times L_{2,\mu}(\R_{\leq0};L_{2}(\Omega)\times L_{2}(\Omega)^{n})
\]
for each $\mu\leq\rho.$ 
\end{prop}

\begin{proof}
Let $g\in H_{\rho}^{1}(\R_{\leq0};L_{2}(\Omega)\times L_{2}(\Omega)^{n})$.
Then
\begin{align*}
\Gamma_{\rho}g & =\left(\left(\left(\begin{array}{cc}
1 & 0\\
0 & k^{-1}
\end{array}\right)+\partial_{t,\rho}^{-1}\left(\begin{array}{c}
M_{1}(\partial_{t,\rho})\\
0
\end{array}\right)\right)\chi_{\R_{\geq0}}g(0-)\right)(0+)\\
 & =\left(\left(\begin{array}{cc}
1 & 0\\
0 & k^{-1}
\end{array}\right)\chi_{\R_{\geq0}}g(0-)\right)(0+)\\
 & =\left(\begin{array}{cc}
1 & 0\\
0 & k^{-1}
\end{array}\right)g(0-),
\end{align*}
where we have used 
\[
\partial_{t,\rho}^{-1}\left(\begin{array}{c}
M_{1}(\partial_{t,\rho})\\
0
\end{array}\right)\chi_{\R_{\geq0}}g(0-)\in H_{\rho}^{1}(\R;H)
\]
and hence 
\begin{align*}
\left(\partial_{t,\rho}^{-1}\left(\begin{array}{c}
M_{1}(\partial_{t,\rho})\\
0
\end{array}\right)\chi_{\R_{\geq0}}g(0-)\right)(0+) & =\left(\partial_{t,\rho}^{-1}\left(\begin{array}{c}
M_{1}(\partial_{t,\rho})\\
0
\end{array}\right)\chi_{\R_{\geq0}}g(0-)\right)(0-)=0
\end{align*}
by causality. Moreover, 
\begin{align*}
K_{\rho}g & =P_{0}\partial_{t,\rho}M(\partial_{t,\rho})g\\
 & =P_{0}\partial_{t,\rho}\left(\begin{array}{cc}
1 & 0\\
0 & k^{-1}
\end{array}\right)g+\chi_{\R_{\geq0}}(\m)\left(\begin{array}{c}
M_{1}(\partial_{t,\rho})\\
0
\end{array}\right)g\\
 & =\chi_{\R_{\geq0}}(\m)\left(\begin{array}{c}
M_{1}(\partial_{t,\rho})\\
0
\end{array}\right)g,
\end{align*}
since 
\[
\spt\partial_{t,\rho}\left(\begin{array}{cc}
1 & 0\\
0 & k^{-1}
\end{array}\right)g\subseteq\R_{\leq0}
\]
and thus, $P_{0}\partial_{t,\rho}\left(\begin{array}{cc}
1 & 0\\
0 & k^{-1}
\end{array}\right)g=0$ by \prettyref{prop:properties_cut_off} (c). Let now $g$ be an element
of the set on the left hand side of \prettyref{eq:subset_D_g}. Then,
we compute 
\begin{align*}
 & S_{\rho}\left(\delta_{0}\left(\begin{array}{cc}
1 & 0\\
0 & k^{-1}
\end{array}\right)g(0-)-\chi_{\R_{\geq0}}(\m)\left(\begin{array}{c}
M_{1}(\partial_{t,\rho})\\
0
\end{array}\right)g\right)-\chi_{\R_{\geq0}}g(0-)\\
= & S_{\rho}\left(\delta_{0}\left(\begin{array}{cc}
1 & 0\\
0 & k^{-1}
\end{array}\right)g(0-)-\chi_{\R_{\geq0}}(\m)\left(\begin{array}{c}
M_{1}(\partial_{t,\rho})\\
0
\end{array}\right)g-\right.\\
 & \left.\quad-\partial_{t,\rho}\left(\begin{array}{cc}
1 & 0\\
0 & k^{-1}
\end{array}\right)\chi_{\R_{\ge0}}g(0-)-\left(\begin{array}{c}
M_{1}(\partial_{t,\rho})\\
0
\end{array}\right)\chi_{\R_{\geq0}}g(0-)-\chi_{\R_{\geq0}}Ag(0-)\right)\\
= & -S_{\rho}\left(\chi_{\R_{\geq0}}(\m)\left(\begin{array}{c}
M_{1}(\partial_{t,\rho})\\
0
\end{array}\right)(g+\chi_{\R_{\geq0}}g(0-))\right)-S_{\rho}\left(\chi_{\R_{\geq0}}Ag(0-)\right).
\end{align*}
We now treat both terms separately. We note that 
\[
\left(\chi_{\R_{\geq0}}(\m)c_{j}\tau_{-h_{j}}f_{j}\right)(t)=\begin{cases}
c_{j}f_{j}(t-h_{j}) & \text{ if }t\geq0,\\
0 & \text{ otherwise}
\end{cases}
\]
for $f_{j}\in H_{\rho}^{1}(\R;L_{2}(\Omega))$ and thus, $\chi_{\R_{\geq0}}(\m)c_{j}\tau_{-h_{j}}f_{j}\in H_{\rho}^{1}(\R;L_{2}(\Omega))$
if $f_{j}(-h_{j})=0.$ Thus, by the constraints on $g$, we infer
that 
\[
\chi_{\R_{\geq0}}(\m)\left(\begin{array}{c}
M_{1}(\partial_{t,\rho})\\
0
\end{array}\right)(g+\chi_{\R_{\geq0}}g(0-))\in H_{\rho}^{1}(\R;L_{2}(\Omega)\times L_{2}(\Omega)^{n}).
\]
Thus, we are left to consider the last term. By assumption, we find
$x\in\dom(A)$ with $Ag(0-)=\left(\begin{array}{cc}
1 & 0\\
0 & k^{-1}
\end{array}\right)x$ and hence, 
\begin{align*}
S_{\rho}\left(\chi_{\R_{\geq0}}Ag(0-)\right) & =\partial_{t,\rho}^{-1}S_{\rho}\left(\partial_{t,\rho}\left(\begin{array}{cc}
1 & 0\\
0 & k^{-1}
\end{array}\right)\chi_{\R_{\geq0}}x\right)\\
 & =\partial_{t,\rho}^{-1}\left(\chi_{\R_{\geq0}}x-S_{\rho}\left(\left(\left(\begin{array}{c}
M_{1}(\partial_{t,\rho})\\
0
\end{array}\right)+A\right)\chi_{\R_{\geq0}}x\right)\right)\\
 & \in H_{\rho}^{1}(\R;L_{2}(\Omega)\times L_{2}(\Omega)^{n}),
\end{align*}
which proves the claim. 
\end{proof}
We conclude this section by proving that the associated semigroup
can be extended to $X_{\rho}^{\mu}$ for each $\mu\leq\rho.$ 
\begin{thm}
Let $\rho>s_{0}(M,A)$ and let $T^{\rho}$ denote the associated semigroup
with $(M,A)$ on $D_{\rho}.$ Then for large enough $\rho,$ $T^{\rho}$
extends to a $C_{0}$-semigroup on $\left(L_{2}(\Omega)\times L_{2}(\Omega)^{n}\right)\times L_{2,\mu}(\R_{\leq0};L_{2}(\Omega)\times L_{2}(\Omega)^{n})$
for each $\mu\leq\rho.$ 
\end{thm}

\begin{proof}
The proof will be done by a perturbation argument. For doing so, we
consider the evolutionary problem associated with $\left(E,A\right)$,
where 
\[
E\coloneqq\left(\begin{array}{cc}
1 & 0\\
0 & k^{-1}
\end{array}\right).
\]
We note that this problem is well-posed with $s_{0}\left(E,A\right)=0$
(compare the proof of \prettyref{prop:well_posed_delay}). We denote
the associated semigroup by $\tilde{T}^{\rho}$. By \prettyref{prop:IV_DAE}
we know that the closure of the initial value space for $\tilde{T}^{\rho}$
is given by
\[
\overline{\{x\in\dom(A)\,;\,E^{-1}Ax\in\dom(A)\}}=L_{2}(\Omega)\times L_{2}(\Omega)^{n}.
\]
Moreover, by \prettyref{thm:HY_DAE} $\tilde{T}^{\rho}$ extends to
a $C_{0}$-semigroup on $\left(L_{2}(\Omega)\times L_{2}(\Omega)^{n}\right)\times L_{2,\mu}(\R_{\leq0};L_{2}(\Omega)\times L_{2}(\Omega)^{n})$
if and only if 
\[
\|\left((\lambda E+A)^{-1}E\right)^{n}\|\leq\frac{M}{(\lambda-\omega)^{n}}\quad(\lambda>\omega,n\in\N)
\]
for some $M\geq1,\,\omega\geq\rho$. We note that $E$ is selfadjoint
and strictly positive definite and thus, 
\[
(\lambda E+A)^{-1}=\sqrt{E^{-1}}\left(\lambda+\sqrt{E^{-1}}A\sqrt{E^{-1}}\right)^{-1}\sqrt{E^{-1}}.
\]
The latter gives 
\[
\left((\lambda E+A)^{-1}E\right)^{n}=\sqrt{E^{-1}}\left(\lambda+\sqrt{E^{-1}}A\sqrt{E^{-1}}\right)^{-n}\sqrt{E}
\]
for each $n\in\N$. Since $A$ is skew-selfadjoint, so is $\sqrt{E^{-1}}A\sqrt{E^{-1}}$
and thus, 
\[
\|\left((\lambda E+A)^{-1}E\right)^{n}\|\leq\frac{\|\sqrt{E}\|\|\sqrt{E^{-1}}\|}{\lambda^{n}}\quad(\lambda>0,n\in\N)
\]
and hence, $\tilde{T}^{\rho}$ extends to a bounded $C_{0}$-semigroup
on $\left(L_{2}(\Omega)\times L_{2}(\Omega)^{n}\right)\times L_{2,\mu}(\R_{\leq0};L_{2}(\Omega)\times L_{2}(\Omega)^{n})$.
Now we come to the semigroup $T^{\rho}.$ We will prove that $T_{1}^{\rho}:D_{\rho}\subseteq\left(L_{2}(\Omega)\times L_{2}(\Omega)^{n}\right)\times L_{2,\mu}(\R_{\leq0};L_{2}(\Omega)\times L_{2}(\Omega)^{n})\to C_{\omega}(\R_{\geq0};L_{2}(\Omega)\times L_{2}(\Omega)^{n})$
is bounded for some $\omega\in\R$, which would imply the claim by
\prettyref{prop:T1_suffices}. Let $(g(0-),g)\in D_{\rho}.$ We then
have, using the formulas in \prettyref{prop:His_delay},
\begin{align*}
T_{1}^{\rho}(g(0-),g) & =\left(\partial_{t,\rho}M(\partial_{t,\rho})+A\right)^{-1}\left(\delta_{0}Eg(0-)-\chi_{\R_{\geq0}}(\m)\left(\begin{array}{c}
M_{1}(\partial_{t,\rho})\\
0
\end{array}\right)g\right)\\
 & =\left(\partial_{t,\rho}E+A+\left(\begin{array}{c}
M_{1}(\partial_{t,\rho})\\
0
\end{array}\right)\right)^{-1}\left(\delta_{0}Eg(0-)\right)-\\
 & \quad-\left(\partial_{t,\rho}M(\partial_{t,\rho})+A\right)^{-1}\left(\chi_{\R_{\geq0}}(\m)\left(\begin{array}{c}
M_{1}(\partial_{t,\rho})\\
0
\end{array}\right)g\right).
\end{align*}
The first term in the latter expression can be rewritten as 
\[
\left(\partial_{t,\rho}E+A+\left(\begin{array}{c}
M_{1}(\partial_{t,\rho})\\
0
\end{array}\right)\right)^{-1}\left(\delta_{0}Eg(0-)\right)=\left(1+\left(\partial_{t,\rho}E+A\right)^{-1}\left(\begin{array}{c}
M_{1}(\partial_{t,\rho})\\
0
\end{array}\right)\right)^{-1}\tilde{T}_{1}^{\rho}\left(\delta_{0}Eg(0-)\right).
\]
Now, since $\tilde{T}_{1}^{\rho}\left(\delta_{0}Eg(0-)\right)\in L_{2,\rho}(\R;L_{2}(\Omega)\times L_{2}(\Omega)^{n})$
and $\|M_{1}(\partial_{t,\rho})\|\to0$ as $\rho\to\infty,$ we infer
that 
\[
\left(\partial_{t,\rho}E+A+\left(\begin{array}{c}
M_{1}(\partial_{t,\rho})\\
0
\end{array}\right)\right)^{-1}\left(\delta_{0}Eg(0-)\right)\in L_{2,\rho}(\R;L_{2}(\Omega)\times L_{2}(\Omega)^{n})
\]
 for $\rho$ large enough by the Neumann series. Since clearly 
\[
L_{2}(\Omega)\times L_{2}(\Omega)^{n}\ni x\mapsto\left(\partial_{t,\rho}E+A+\left(\begin{array}{c}
M_{1}(\partial_{t,\rho})\\
0
\end{array}\right)\right)^{-1}\left(\delta_{0}Ex\right)\in H_{\rho}^{-1}(\R;L_{2}(\Omega)\times L_{2}(\Omega)^{n})
\]
is bounded, we obtain 
\[
\left\Vert \left(\partial_{t,\rho}E+A+\left(\begin{array}{c}
M_{1}(\partial_{t,\rho})\\
0
\end{array}\right)\right)^{-1}\left(\delta_{0}Eg(0-)\right)\right\Vert _{L_{2,\rho}(\R;L_{2}(\Omega)\times L_{2}(\Omega)^{n})}\leq C\|g(0-)\|_{L_{2}(\Omega)\times L_{2}(\Omega)^{n}}
\]
for some $C\geq0$ by the closed graph theorem. Hence, 
\begin{align*}
 & \|T_{1}^{\rho}(g(0-),g)\|_{L_{2,\rho}(\R;L_{2}(\Omega)\times L_{2}(\Omega)^{n})}\\
 & \leq C\|g(0-)\|_{L_{2}(\Omega)\times L_{2}(\Omega)^{n}}+\left\Vert \left(\partial_{t,\rho}M(\partial_{t,\rho})+A\right)^{-1}\left(\chi_{\R_{\geq0}}(\m)\left(\begin{array}{c}
M_{1}(\partial_{t,\rho})\\
0
\end{array}\right)g\right)\right\Vert _{L_{2,\rho}(\R;L_{2}(\Omega)\times L_{2}(\Omega)^{n})}\\
 & \leq C\|g(0-)\|_{L_{2}(\Omega)\times L_{2}(\Omega)^{n}}+C_{1}\|g\|_{L_{2,\rho}(\R;L_{2}(\Omega)\times L_{2}(\Omega)^{n})}\\
 & \leq\tilde{C}\|(g(0-),g)\|_{X_{\rho}^{\mu}}
\end{align*}
for suitable $C_{1},\tilde{C}\geq0.$ Thus, 
\[
T_{1}^{\rho}:D_{\rho}\subseteq X_{\rho}^{\mu}\to L_{2,\rho}(\R_{\geq0};L_{2}(\Omega)\times L_{2}(\Omega)^{n})
\]
is bounded and hence extends to a bounded operator on $X_{\rho}^{\mu}$.
Moreover, for $f\in C_{c}^{\infty}(\R_{\geq0};L_{2}(\Omega)\times L_{2}(\Omega)^{n})$
we may estimate 
\begin{align*}
\left\Vert \left((\partial_{t,\rho}E+A)^{-1}f\right)(t)\right\Vert  & =\left\Vert \int_{0}^{t}\tilde{T}_{\rho}^{(1)}(t-s)E^{-1}f(s)\d s\right\Vert \\
 & \leq M\|E^{-1}\|\int_{0}^{t}\|f(s)\|\d s\\
 & \leq\frac{M\|E^{-1}\|}{\sqrt{2\rho}}\|f\|_{L_{2,\rho}(\R;L_{2}(\Omega)\times L_{2}(\Omega)^{n})}\e^{\rho t}
\end{align*}
for each $t\geq0$, which proves that
\[
(\partial_{t,\rho}E+A)^{-1}:L_{2,\rho}(\R_{\geq0};L_{2}(\Omega)\times L_{2}(\Omega)^{n})\to C_{\rho}(\R_{\geq0};L_{2}(\Omega)\times L_{2}(\Omega)^{n})
\]
is bounded. Now, let $(x,g)\in X_{\rho}^{\mu}$ and set $u\coloneqq T_{1}^{\rho}(x,g)\in L_{2,\rho}(\R_{\geq0};L_{2}(\Omega)\times L_{2}(\Omega)^{n})$.
Then 
\[
(\partial_{t,\rho}E+A)u=\delta_{0}Ex-\chi_{\R_{\geq0}}(\m)\left(\begin{array}{c}
M_{1}(\partial_{t,\rho})\\
0
\end{array}\right)g-\left(\begin{array}{c}
M_{1}(\partial_{t,\rho})\\
0
\end{array}\right)u
\]
and hence, we derive that 
\begin{align*}
u & =(\partial_{t,\rho}E+A)^{-1}\left(\delta_{0}Ex-\chi_{\R_{\geq0}}(\m)\left(\begin{array}{c}
M_{1}(\partial_{t,\rho})\\
0
\end{array}\right)g-\left(\begin{array}{c}
M_{1}(\partial_{t,\rho})\\
0
\end{array}\right)u\right)\\
 & =\tilde{T}_{1}^{\rho}(x,g)-(\partial_{t,\rho}E+A)^{-1}\left(\chi_{\R_{\geq0}}(\m)\left(\begin{array}{c}
M_{1}(\partial_{t,\rho})\\
0
\end{array}\right)g+\left(\begin{array}{c}
M_{1}(\partial_{t,\rho})\\
0
\end{array}\right)u\right)\\
 & \in C_{\rho}(\R_{\geq0};L_{2}(\Omega)\times L_{2}(\Omega)^{n})
\end{align*}
and hence, again by the closed graph theorem 
\[
T_{1}^{\rho}:X_{\rho}^{\mu}\to C_{\rho}(\R_{\geq0};L_{2}(\Omega)\times L_{2}(\Omega)^{n})
\]
is bounded. 
\end{proof}


\begin{thebibliography}{10}

\bibitem{Arendt1987}
W.~{Arendt}.
\newblock {Vector-valued Laplace transforms and Cauchy problems.}
\newblock {\em {Isr. J. Math.}}, 59:327--352, 1987.

\bibitem{ABHN_2011}
W.~{Arendt}, C.~J. {Batty}, M.~{Hieber}, and F.~{Neubrander}.
\newblock {\em {Vector-valued Laplace transforms and Cauchy problems. 2nd ed.}}
\newblock Basel: Birkh\"auser, 2011.

\bibitem{Batkai_2005}
A.~B{\'a}tkai and S.~Piazzera.
\newblock {\em {Semigroups for delay equations.}}
\newblock {Research Notes in Mathematics 10. Wellesley, MA: A K Peters. xii},
  2005.

\bibitem{DiestelUhl}
J.~{Diestel} and J.~{Uhl}.
\newblock {Vector measures.}
\newblock {Mathematical Surveys. No.15. Providence, R.I.: American Mathematical
  Society (AMS). XIII.}, 1977.

\bibitem{engel2000one}
K.~J. Engel and R.~Nagel.
\newblock {\em {One-parameter semigroups for linear evolution equations}}.
\newblock Graduate texts in mathematics. Springer, 2000.

\bibitem{Foures1955}
Y.~{Four\`es} and I.~{Segal}.
\newblock {Causality and analyticity.}
\newblock {\em {Trans. Am. Math. Soc.}}, 78:385--405, 1955.

\bibitem{hille1957functional}
E.~Hille and R.~S. Phillips.
\newblock {\em {Functional analysis and semi-groups}}.
\newblock Colloquium Publications - American Mathematical Society. American
  Mathematical Society, 1957.

\bibitem{Kalauch2011}
A.~{Kalauch}, R.~{Picard}, S.~{Siegmund}, S.~{Trostorff}, and M.~{Waurick}.
\newblock {A Hilbert space perspective on ordinary differential equations with
  memory term.}
\newblock {\em {J. Dyn. Differ. Equations}}, 26(2):369--399, 2014.

\bibitem{Mehrmann2006}
P.~{Kunkel} and V.~{Mehrmann}.
\newblock {\em {Differential-algebraic equations. Analysis and numerical
  solution.}}
\newblock Z\"urich: European Mathematical Society Publishing House, 2006.

\bibitem{Paley_Wiener}
R.~E. {Paley} and N.~{Wiener}.
\newblock {Fourier transforms in the complex domain.}
\newblock {(Am. Math. Soc. Colloq. Publ. 19) New York: Am. Math. Soc. VIII},
  1934.

\bibitem{Pazy1983}
A.~{Pazy}.
\newblock {Semigroups of linear operators and applications to partial
  differential equations.}
\newblock {Applied Mathematical Sciences, 44. New York etc.: Springer-Verlag.
  VIII}, 1983.

\bibitem{Picard}
R.~Picard.
\newblock {A structural observation for linear material laws in classical
  mathematical physics.}
\newblock {\em Math. Methods Appl. Sci.}, 32(14):1768--1803, 2009.

\bibitem{Picard_McGhee}
R.~Picard and D.~McGhee.
\newblock {\em {Partial differential equations. A unified Hilbert space
  approach.}}
\newblock {de Gruyter Expositions in Mathematics 55. Berlin: de Gruyter.
  xviii}, 2011.

\bibitem{Picard2013_fractional}
R.~{Picard}, S.~{Trostorff}, and M.~{Waurick}.
\newblock {On evolutionary equations with material laws containing fractional
  integrals.}
\newblock {\em {Math. Methods Appl. Sci.}}, 38(15):3141--3154, 2015.

\bibitem{Picard2013_nonauto}
R.~Picard, S.~Trostorff, M.~Waurick, and M.~Wehowski.
\newblock {On Non-autonomous Evolutionary Problems.}
\newblock {\em J. Evol. Equ.}, 13(4):751--776, 2013.

\bibitem{picard1989hilbert}
R.~H. Picard.
\newblock {\em {Hilbert space approach to some classical transforms}}.
\newblock Pitman research notes in mathematics series. Longman Scientific \&
  Technical, 1989.

\bibitem{rudin1987real}
W.~Rudin.
\newblock {\em {Real and complex analysis}}.
\newblock Mathematics series. McGraw-Hill, 1987.

\bibitem{Trostorff2012_NA}
S.~Trostorff.
\newblock {An alternative approach to well-posedness of a class of differential
  inclusions in Hilbert spaces.}
\newblock {\em Nonlinear Anal., Theory Methods Appl., Ser. A, Theory Methods},
  75(15):5851--5865, 2012.

\bibitem{Trostorff2012_nonlin_bd}
S.~Trostorff.
\newblock {Autonomous Evolutionary Inclusions with Applications to Problems
  with Nonlinear Boundary Conditions.}
\newblock {\em Int. J. Pure Appl. Math.}, 85(2):303--338, 2013.

\bibitem{Trostorff2013_bd_maxmon}
S.~{Trostorff}.
\newblock {A characterization of boundary conditions yielding maximal monotone
  operators.}
\newblock {\em {J. Funct. Anal.}}, 267(8):2787--2822, 2014.

\bibitem{Trostorff2015_secondorder}
S.~{Trostorff}.
\newblock {Exponential stability for second order evolutionary problems.}
\newblock {\em {J. Math. Anal. Appl.}}, 429(2):1007--1032, 2015.

\bibitem{Trostorff2012_integro}
S.~Trostorff.
\newblock {On Integro-Differential Inclusions with Operator-valued Kernels.}
\newblock {\em Math. Methods Appl. Sci.}, 38(5):834--850, 2015.
\newblock \href {http://dx.doi.org/10.1002/mma.3111}
  {\path{doi:10.1002/mma.3111}}.

\bibitem{Trostorff_habil}
S.~Trostorff.
\newblock {\em Exponential Stability and Initial Value Problems for
  Evolutionary Equations}.
\newblock Habilitation thesis, TU Dresden, 2018.
\newblock URL: \url{https://nbn-resolving.org/urn:nbn:de:bsz:14-qucosa-236494}.

\bibitem{Trostorff2018}
S.~Trostorff.
\newblock Well-posedness for a general class of differential inclusions.
\newblock Technical report, TU Dresden, 2018.
\newblock arXiv: 1808.00224, submitted.

\bibitem{Trostorff2019_DAE_semigroup}
S.~Trostorff.
\newblock Semigroups associated with differential-algebraic equations.
\newblock Technical report, CAU Kiel, 2019.
\newblock arXiv:1905.11197.

\bibitem{Trostorff_DAE_higherindex_2018}
S.~Trostorff and M.~Waurick.
\newblock On higher index differential-algebraic equations in infinite
  dimensions.
\newblock In {\em The diversity and beauty of applied operator theory}, volume
  268 of {\em Oper. Theory Adv. Appl.}, pages 477--486.
  Birkh\"{a}user/Springer, Cham, 2018.

\bibitem{Trostorff_DAE2019}
S.~{Trostorff} and M.~{Waurick}.
\newblock {On differential-algebraic equations in infinite dimensions.}
\newblock {\em {J. Differ. Equations}}, 266(1):526--561, 2019.

\bibitem{Trostorff2013_nonautoincl}
S.~{Trostorff} and M.~{Wehowski}.
\newblock {Well-posedness of non-autonomous evolutionary inclusions.}
\newblock {\em {Nonlinear Anal., Theory Methods Appl., Ser. A, Theory
  Methods}}, 101:47--65, 2014.

\bibitem{Waurick2015_nonauto}
M.~{Waurick}.
\newblock {On non-autonomous integro-differential-algebraic evolutionary
  problems.}
\newblock {\em {Math. Methods Appl. Sci.}}, 38(4):665--676, 2015.

\bibitem{Weidmann}
J.~{Weidmann}.
\newblock {\em {Linear operators in Hilbert spaces. Transl. by Joseph
  Sz\"ucs.}}, volume~68.
\newblock Springer, New York, NY, 1980.

\bibitem{Weiss1991}
G.~{Weiss}.
\newblock {Representation of shift-invariant operators on $L\sp 2$ by
  $H\sp{\infty}$ transfer functions: An elementary proof, a generalization to
  $L\sp p$, and a counterexample for $L\sp{\infty}$.}
\newblock {\em {Math. Control Signals Syst.}}, 4(2):193--203, 1991.

\bibitem{Widder1934}
D.~{Widder}.
\newblock {The inversion of the Laplace integral and the related moment
  problem.}
\newblock {\em {Trans. Am. Math. Soc.}}, 36:107--200, 1934.

\end{thebibliography}
\end{document}